\documentclass[11pt,regno]{amsart}
\usepackage{amsmath,amsthm,amsfonts,amssymb,amscd,overpic,mathrsfs}
\usepackage{euscript,graphicx,color}
\usepackage{epstopdf,rotating}
\newtheorem{mtheorem}{Theorem}
\newtheorem{mcorollary}[mtheorem]{Corollary}

\newtheorem{theorem}{Theorem}[section]
\newtheorem{lemma}[theorem]{Lemma}
\newtheorem{claim}[theorem]{Claim}

\newtheorem{corollary}[theorem]{Corollary}
\newtheorem{proposition}[theorem]{Proposition}

\theoremstyle{definition}
\newtheorem{definition}[theorem]{Definition}

\newtheorem{remark}[theorem]{Remark}
\newtheorem*{remark*}{Remark}

\numberwithin{equation}{section}

\newcommand{\eqdef}{\stackrel{\scriptscriptstyle\rm def}{=}}

\DeclareMathOperator{\card}{card}
\DeclareMathOperator{\dist}{Dist}

\DeclareMathOperator{\Mod}{Mod}
\DeclareMathOperator{\conv}{conv}
\DeclareMathOperator{\cconv}{\overline{\conv}}

\DeclareMathOperator{\Diff}{Diff}

\DeclareMathOperator{\Wall}{Wall}

\def\cRTPH{\mathbf{RTPH}}
\def\cORTPH{\mathbf{MB}}

\def\Diff{\mathrm{Diff}^1}

\def\varkappa{\kappa}

\def\bZ{\mathbb{Z}}
\def\bR{\mathbb{R}}

\def\cC{\EuScript{C}}

\def\cU{\mathcal{U}}
\def\cF{\mathcal{F}}

\def\cV{\EuScript{V}}
\def\cW{\EuScript{W}}
\def\cM{\EuScript{M}}

\def\cN{\EuScript{N}}

\def\cL{\EuScript{L}}

\def\loc{{\rm loc}}
\def\ut{{\rm u}}
\def\st{{\rm s}}
\def\ss{{\rm ss}}
\def\cs{{\rm cs}}
\def\cu{{\rm cu}}

\def\c{{\rm c}}
\def\s{{\rm s}}
\def\u{{\rm u}}
\def\cu{{\rm cu}}

\def\uu{{\rm uu}}
\newcommand{\eps}{\varepsilon}
\def\safe{\tau}

\DeclareMathSymbol{\varnothing}{\mathord}{AMSb}{"3F}
\renewcommand{\emptyset}{\varnothing}

\begin{document}

\title[Weak$\ast$-entropy approximation of nonhyperbolic measures]{Weak$\ast$ and entropy approximation\\ of nonhyperbolic measures:\\ a geometrical approach}
\author[L.~J.~D\'iaz]{Lorenzo J. D\'\i az}
\address{Departamento de Matem\'atica PUC-Rio, Marqu\^es de S\~ao Vicente 225, G\'avea, Rio de Janeiro 22451-900, Brazil}
\email{lodiaz@mat.puc-rio.br}
\author[K.~Gelfert]{Katrin Gelfert}
\address{Instituto de Matem\'atica, Universidade Federal do Rio de Janeiro, Av. Athos da Silveira Ramos 149, Cidade Universit\'aria - Ilha do Fund\~ao, Rio de Janeiro 21945-909,  Brazil}\email{gelfert@im.ufrj.br}
\author[B.~Santiago]{Bruno Santiago}
\address{Instituto de Matem\'atica e Estat\'istica, Universidade Federal Fluminense, Rua Professor Marcos Waldemar de Freitas Reis, s/n, Bloco H - Campus do Gragoat\'a S\~ao Domingos, Niter\'oi 24210-201, Brazil}
\email{brunosantiago@id.uff.br}

\begin{abstract}
We study $C^1$-robustly transitive and nonhyperbolic diffeomorphisms having a partially hyperbolic splitting with one-dimensional central bundle whose strong un-/stable foliations are both minimal. {In dimension $3$, an important class of examples of such systems is given by those with a simple closed periodic curve tangent to the central bundle.} We prove that there is a $C^1$-open and dense subset of such diffeomorphisms such that every nonhyperbolic ergodic measure (i.e. with zero central exponent) can be approximated in the weak$\ast$ topology and in entropy by measures supported in basic sets with positive (negative) central Lyapunov exponent. 
Our method also allows to show how entropy changes across measures with  central Lyapunov exponent close to zero. 
We also prove that any nonhyperbolic ergodic measure is in the intersection of the convex hulls of the measures with positive central exponent and with negative central exponent.
 \end{abstract}

\begin{thanks}{
This study was financed in part by the Coordenação de Aperfeiçoamento de Pessoal de Nível Superior, Brasil (CAPES), Finance code 001 and also partially by
CNE-Faperj and CNPq grants (Brazil). The authors thank the referees for the comments and A. Tahzibi for calling attention to \cite{Tah:}.} 
\end{thanks}

\maketitle
\tableofcontents

\section{Introduction}
{Consider a boundaryless Riemannian compact manifold $M$ and its space $\Diff(M)$ of $C^1$-diffeomorphisms endowed with  the uniform topology. 
We consider the $C^1$-open subset  of $\Diff(M)$,
denoted by $\cRTPH^1(M)$, formed by diffeomorphisms $f$ 
 with  a $C^1$-neighborhood $\cV_f$  whose elements
satisfy properties (H1)--(H3) that we proceed to describe: 
\begin{itemize}
\item[(H1)] Every diffeomorphism $g$ in $\cV_f$ is nonhyperbolic.
\item[(H2)] There is a partially hyperbolic splitting $TM=E^\ss \oplus E^\c\oplus E^\uu$ with three non-trivial bundles such that $E^\ss$ is uniformly contracting, $E^\c$ is one-dimensional, and $E^\uu$ is uniformly expanding.
%, and
%\item[(H3)] 
%there is a periodic simple closed curve tangent to $E^\c$.
\end{itemize}
To state hypothesis (H3), we first recall that by partial hyperbolicity, there exist invariant foliations $\cF^\ss$ and $\cF^\uu$ tangent to $E^\ss$ and $E^\uu$ and called  \emph{strong stable} and \emph{strong unstable foliations}, respectively (see \cite{HirPugShu:77}). 
\begin{itemize}
\item[(H3)]  The strong stable and the strong unstable foliations
of any  $g\in \cV_f$ are both \emph{minimal} (that is, every leaf of the foliation is dense in the whole space).
\end{itemize}
}

{Recall that a diffeomorphism is \emph{transitive} if it has a dense orbit and is \emph{$C^1$-robustly transitive} if it has a $C^1$-neighborhood consisting of transitive diffeomorphisms. 
Since the minimality of any strong foliation implies transitivity, condition (H3) implies that every diffeomorphism in  $\cRTPH^1(M)$ is transitive. Since (H3) requires this property in a neighborhood, every $f\in\cRTPH^1(M)$ is robustly transitive.
}

{To comment on our hypotheses, while (H1) and (H2) are quite natural, (H3) may at first seem to be rather restrictive. To describe a natural setting where $f$ satisfies the latter is a bit more elaborate and relies on the existence of a simple closed periodic%
\footnote{That is, there exists $n\ge1$ such that $f^n(\gamma_f)=\gamma_f$.} 
curve $\gamma_f$ tangent to $E^\c$. Since partially hyperbolic splittings have well defined continuations and the curve $\gamma_f$ is normally hyperbolic, it has well defined continuations in a $C^1$-neighborhood of $f$ (see  \cite{HirPugShu:77}).  Note that the existence of a closed periodic curve tangent to $E^\c$ immediately prevents hyperbolicity. The main examples of robustly transitive diffeomorphisms having simple closed periodic curves fall into two classes: those having an invariant foliation tangent to $E^\c$ consisting of circles 
(see \cite{Shu:71b,GorIly:99,BonDia:96}) and those having simultaneously closed and non-closed leaves tangent to $E^\c$. Examples of the latter are appropriate perturbations of the time-one map of a transitive Anosov flow \cite{BonDia:96} {and a certain class of diffeomorphisms  in \cite{BonGogPot:16} (involving a so-called Dehn twist and the time-one map of a hyperbolic geodesic flow).} 
}

{To return to the discussion of simultaneous minimality of both strong foliations, first assume that $\dim M=3$ and that $\cU$ is an open set of $\Diff(M)$ consisting of transitive diffeomorphisms $f$ satisfying (H2) and each having some closed periodic curve $\gamma_f$ tangent to $E^\c$ (thus satisfying (H1)). In this setting, by \cite{BonDiaUre:02}  there is a $C^1$-open and -dense subset of $\cU$ consisting of diffeomorphisms for which both foliations are \emph{minimal} and hence satisfy (H3). For examples in higher dimensions, as recently communicated \cite{Tah:}, robustly transitive perturbations of the time-one maps of Anosov flows (in any dimension) also provide examples of diffeomorphisms having simultaneously minimal foliations.
}

%Note also that transitivity and the existence of a periodic closed curve together imply that the neighborhood of $f$ consists of nonhyperbolic diffeomorphisms. 

%There are several types of examples of diffeomorphisms in  $\cRTPH^1(M)$:
%those having an invariant foliation tangent to $E^\c$ consisting of circles 
%(see \cite{Shu:71b,GorIly:99}) and those having simultaneously closed and non-closed leaves tangent to $E^\c$ (the main example are some perturbations of the time-one map of a transitive Anosov flow, \cite{BonDia:96}). 

{Note that there is an important class of nonhyperbolic partially hyperbolic robustly transitive systems, called DA-diffeomorphisms \cite{Man:78}, which \emph{a priori} do not fall into $\cRTPH^1(M)$ because they do not have closed curves tangent to the central bundle $E^\c$ and hence, so far, it is unknown if (H3) is satisfied.}

{
The next definition involves the notion of a blender-horseshoe, see Section~\ref{sec:blennn} for the precise definition and discussion.}

{
\begin{definition}[The set $\cORTPH^1(M)$]
The set\footnote{This notation refers to \emph{minimality} and existence of \emph{blender-horseshoes}.} 
$\cORTPH^1(M)$ is the subset of  $\cRTPH^1(M)$ consisting of diffeomorphisms with a
pair of  \emph{blender-horseshoes} (one contracting in the central direction and one expanding in the central direction).
\end{definition}}

\begin{remark}[Properties of $\cORTPH^1(M)$]
{
Conditions (H1) and (H2) imply that the set $\cORTPH^1(M)$  is  $C^1$-open and $C^1$-dense in
 $\cRTPH^1(M)$,  see  Proposition \ref{p.blenderopendense}. In fact, to get such blender-horseshoes, hypothesis (H3) is not used at all and, indeed, the existence of the blender-horseshoes is, besides some geometrical hypothesis, the (implicit) key element in \cite{BonDiaUre:02} to prove the minimality of the foliations  (even though the term ``blender-horseshoe'' was only coined later~\cite{BonDia:12}). We will explore the dynamics of these blender-horseshoes, see Section~\ref{sec:blennn}, which will also be an important ingredient in our constructions. For details see Proposition~\ref{p.blenderopendense} and Remark~\ref{r.notation}. 
 }

{
 Assume now that  $M$  has dimension three and consider the set
$\mathbf{RTC}^1(M)$
of robustly transitive diffeomorphisms of $M$ having a closed simple periodic curve
and a partially hyperbolic splitting with three bundles.
Then the set $\mathbf{MB}^1(M)$ is $C^1$-open and $C^1$-dense in
$\mathbf{RTC}^1(M)$, see \cite{BonDiaUre:02}.
 }\end{remark}

\begin{remark}[Essential hypotheses]
{
The proofs of our results do not involve any perturbation. 
%The hypotheses stated as above on one hand just aim to show that we are indeed dealing with a large class of systems. 
The essential hypotheses we do use for every $f$ under consideration are the following: 
\begin{itemize}
\item partial hyperbolicity with splitting $TM=E^\ss\oplus E^\c\oplus E^\uu$ and with one-dimensional center,
\item existence of a pair of blender-horseshoes, one contracting in the central direction and one expanding in the central direction,
\item minimality of both strong foliations.
\end{itemize}
}

{
Note that the simultaneous existence of blender-horseshoes of different type implies nonhyperbolicity. Further, minimality implies transitivity.
}
		
	{Observe that the robustness of the above properties comes along naturally. Indeed, partial hyperbolicity and existence of blender-horseshoes are both robust properties, while \emph{a priori} the minimality of the strong foliations is not. However, the existence of blender-horseshoes forces the robustness of minimality (this is indeed the heart of the proof in \cite{BonDiaUre:02}).}
\end{remark}

Nonhyperbolicity is closely related to the existence of zero Lyapunov exponents.  Given $f\in \Diff (M)$, a point $x\in M$ is  \emph{Lyapunov regular} if there are a positive integer $s(x)$, numbers $\chi_1(x)<\ldots<\chi_{s(x)}(x)$, called the \emph{Lyapunov exponents} of $x$, and a $Df$-invariant splitting $T_xM=\oplus_{i=1}^{s(x)}F_x^i$ such that for all $i=1,\ldots,s(x)$ and $v\in \cF^i_x$, $v\ne0$, we have
\[
	\lim_{n\to\pm\infty}\frac1n\log\,\lVert Df^n_x(v)\rVert
	= \chi_i(x).
\]
Note that in our partially hyperbolic setting there is some $\ell$ such that $\cF^\ell=E^\c$ and we denote the corresponding Lyapunov exponent by $\chi^\c$.

We denote by $\cM(f)$ the set of $f$-invariant probability measures of $f$ and by
$\cM_{\rm erg}(f)$ the subset of ergodic measures. We equip the space $\cM(f)$ with the weak$\ast$ topology.  Given $\mu\in \cM_{\rm erg}(f)$,  Oseledets' multiplicative ergodic theorem \cite{Ose:68} claims that the set of Lyapunov regular points has full measure and $s(\cdot)=s(\mu)$ and $\chi_i(\cdot)=\chi_i(\mu)$, $i=1,\ldots,s(\mu)$, are constant $\mu$-almost everywhere. The latter numbers are called the \emph{Lyapunov exponents} of $\mu$. If $\chi^\c(\mu)=0$ then $\mu$ is called \emph{nonhyperbolic}. Note that in our setting the other  exponents of $\mu$ are nonzero. We denote by $\cM_{{\rm erg},0}(f)$ the subset of $\cM_{\rm erg}(f)$ of nonhyperbolic measures.
Thus, the occurrence of a zero exponent is related to the central direction only and there is a  natural decomposition
\begin{equation*}%}\label{eq:ergdecompo}
	\cM_{\rm erg}(f)
	=\cM_{\rm erg,<0}(f) \cup\cM_{\rm erg,0}(f) \cup\cM_{\rm erg,>0}(f),
\end{equation*}
where $\cM_{\rm erg,<0}(f)$ and $\cM_{\rm erg,>0}(f)$ denote spaces of measures $\mu$ such that $\chi^\c(\mu)<0$ and $\chi^\c(\mu)>0$, respectively.

The exploration of nonhyperbolic ergodic measures is a very active research field which started with the pioneering work in \cite{GorIlyKleNal:05}. Note that by \cite{BocBonDia:16} there is a $C^1$-open and -dense subset of $\cRTPH^1(M)$ consisting of diffeomorphisms $f$ such that $\cM_{{\rm erg},0}(f)$ is nonempty and contains measures with positive entropy. The main focus of this paper is to study how nonhyperbolic ergodic measures insert  in the space of ergodic measures. The main result is how nonhyperbolic measures are weak$\ast$ and in entropy approached by hyperbolic ones which are supported on basic sets. {We also conclude about the topological structure of the space of ergodic measures. For previous results about the denseness of hyperbolic measures supported on periodic orbits, see \cite{BonZha:16}.} Our paper is a continuation of a line of arguments in \cite{DiaGelRam:17,DiaGelRam:17b} where these questions were studied in a skew-product setting and where a general axiomatic framework to attack this problem was introduced, see the discussion after Corollary~\ref{cortheocor:varprinc}. 

\begin{remark}\label{rem:Katok}
By a very classical result,  mainly started by Katok \cite{Kat:80,
KatHas:95}, every \emph{hyperbolic} ergodic measure can be approximated by periodic ones. Here one can consider approximation in the weak$\ast$ topology. Moreover, one can approximate by means of ergodic measures supported on basic sets which converge weak$\ast$ and in entropy, that is, given $\mu$ hyperbolic ergodic, there is a sequence $\Gamma_n$ of basic sets such that $\cM_{\rm erg}(f,\Gamma_n)\to \mu$ in the weak$\ast$ topology and that $h_{\rm top}(f,\Gamma_n)\to h(\mu)$. Katok's result was first shown for  $C^{1+\varepsilon}$ surface diffeomorphisms \cite[Supplement S.5]{KatHas:95}, but extends also to higher-dimensional manifolds and $C^1$- and dominated  diffeomorphisms (see, for example, \cite{Cro:11,LuzSan:13,Gel:16} and references therein and also \cite{WanSun:10}). Below we will present an analogous version for \emph{nonhyperbolic} ergodic measures. 
\end{remark}

Given $f\in \cORTPH^1(M)$ and a hyperbolic set $\Gamma\subset M$ of $f$, denote by $\cM(f,\Gamma)\subset\cM(f)$ the subset of measures supported on $\Gamma$. We define analogously $\cM_{\rm erg}(f,\Gamma)$. We say that a hyperbolic set $\Gamma$ is \emph{central contracting} (\emph{central expanding}) if on $T\Gamma$ the bundle $E^\ss\oplus E^\c$ is stable ($E^\c\oplus E^\uu$ is unstable).
Recall that a set is \emph{basic} if it is compact, $f$-invariant, hyperbolic, locally maximal, and  transitive.

Given a countable dense subset $\{\varphi_j\}_{j\ge1}$ of continuous (nonzero) functions on $M$, recall that 
\[
	D(\nu,\mu)
	\eqdef \sum_{j=1}^\infty2^{-j}\frac{1}{2\lVert\varphi_j\rVert_\infty}\left\lvert\int\varphi_j\,d\nu - \int\varphi_j\,d\mu\right\rvert,
	\quad\lVert\varphi\rVert
	\eqdef \sup\lvert\varphi\rvert,
\]
provides a metric which induces the weak$\ast$ topology on $\cM(f)$.

The following is a consequence of Theorem~\ref{teo:finally} which is stated under the minimal hypotheses which we  require to construct central expanding (contracting) basic sets as stated.

\begin{mtheorem}[Approximation in weak$\ast$ and entropy] \label{t.approx}
For every $f\in\cORTPH^1(M)$ every nonhyperbolic ergodic measure $\mu$ of $f$ has the following properties. For every $\delta>0$ and every $\gamma>0$ there exist a pair of basic sets $\Gamma^-$ being central contracting and $\Gamma^+$ being central expanding such that the topological entropy of $f$ on $\Gamma^\mp$ satisfy 
\[	
	h_{\rm top}(f,\Gamma^\mp) \in  [h(\mu)-\gamma, h(\mu)+\gamma].
\]	
Moreover, every measure $\nu^\mp \in\cM(f,\Gamma^\mp)$  is $\delta$-close to $\mu$. In particular, there are  hyperbolic measures $\nu^\mp\in\cM_{\rm erg}(f,\Gamma^\mp)$  satisfying 
$$
	\chi(\nu^-)\in(-\delta,0) 
	\quad \text{and} \quad  
	\chi(\nu^+)\in(0,\delta)
$$
and
$$
	h(\nu^\mp)\in [h(\mu)-\gamma, h(\mu)+\gamma].
$$
\end{mtheorem}

The program for proving the above result was laid out in~\cite[Section 8.3]{DiaGelRam:17}. The result above is the corresponding version of \cite[Theorem 1]{DiaGelRam:17} (in a step skew-product setting with circle fiber maps) in the present setting. {The main difficulties of this translation are discussed below in Sections \ref{ss:axioo} and \ref{subsec:Idea}.}
During the final preparation of this manuscript, we noticed that a preprint with a similar result was announced in~\cite{YanZha:}.

We have the following straightforward consequence of the above.

\begin{mcorollary}[Restricted variational principles]\label{cortheocor:varprinc}
For every $f\in\cORTPH^1(M)$ 
\[
	h_{\rm top}(f)
	= \sup_{\mu\in \cM_{\rm erg,<0}(f)\cup\cM_{\rm erg,>0}(f)} h(\mu).
\]
\end{mcorollary}

Note that, in contrast to~\cite[Theorem 2]{DiaGelRam:} or~\cite{TahYan:}, in general there are yet no general tools to establish the uniqueness of hyperbolic measures of maximal entropy. See also the results and discussion in~\cite{RodRodTahUre:12}.

Recall that an ergodic measure is \emph{periodic} if it is supported on a periodic orbit. It is a classical result by Sigmund \cite{Sig:74} that periodic measures are dense in $\cM(f,\Gamma)$ for any basic set $\Gamma$, and hence every hyperbolic ergodic measure is approximated by hyperbolic periodic ones. The above result then immediately implies that this is also true for nonhyperbolic ergodic measures. 

\begin{mcorollary}[Periodic approximation]\label{cor:2}
For every $f\in\cORTPH^1(M)$ and every $\mu \in \cM_{\rm erg}(f)$ is approximated by hyperbolic periodic measures. Moreover, every $\mu \in \cM_{\rm erg,0}(f)$ is approximated by periodic measures in $\cM_{\rm erg,<0}(f)$ and in $\cM_{\rm erg,>0}(f)$, respectively. 
\end{mcorollary}

{Let us observe that a similar result was previously obtained in \cite{BonZha:16} assuming minimality of strong foliations and concluding correspondingly about the nature (more precisely the
{\em{index,}} that is, number of negative Lyapunov exponents) of the measures supported on the hyperbolic periodic orbits. }

The following result shows how entropy ``changes across measures with Lyapunov exponent close to zero''. As for Theorem~\ref{t.approx}, it will be an immediate consequence of a Theorem~\ref{teo:stated} correspondingly stated under the minimal hypotheses.

\begin{mtheorem}\label{theo:main3twin}
	For every $f\in\cORTPH^1(M)$  and every $\mu\in \cM_{\rm erg}(f)$ with $\alpha= \chi(\mu) <0$, there is a positive constant $K(f)\ge (\log\,\lVert Df\rVert)^{-1}$ such that for every $\delta>0$, $\gamma>0$, and  $\beta>0$, there is a basic set $\Gamma$ being central expanding such that
\begin{itemize}
\item[1.] its topological entropy satisfies
	\[
		h_{\rm top}(f,\Gamma)
		\ge \frac{h(\mu)}{1+K(f)(\beta+\lvert\alpha\rvert) } - \gamma,
	\]
\item[2.] every $\nu\in \cM_{\rm erg}(f,\Gamma)$ satisfies
	\[
		 \frac \beta {1 + K(f)(\beta+\lvert\alpha\rvert)}  - \delta
		 < \chi(\nu) <
		 \frac \beta {1+\frac{1}{\log \lVert Df\rVert}(\beta+\lvert\alpha\rvert)}+ \delta,
	\]		
and
	\[
		D(\nu,\mu)
		<\frac{K(f)(\beta+\lvert \alpha\rvert)}
			{1+K(f)(\beta+\lvert \alpha\rvert)}
		+\delta.
	\]	
\end{itemize}
The same conclusion is true for $\alpha > 0$ and every $\beta < 0$, changing in the assertion $\beta+|\alpha|$ to $|\beta|+\alpha$.

If $h(\mu)=0$ then $\Gamma$ is a hyperbolic periodic orbit.
\end{mtheorem}

The result above corresponds to \cite[Theorem 5]{DiaGelRam:17}.

\begin{remark}[Continuations in the weak$\ast$ and in entropy of ergodic measures]
A consequence of Theorem~\ref{t.approx} is that for every $f\in \cORTPH^1(M)$ any ergodic measure $\mu$ of $f$ has a
\emph{continuation} in the following sense. Every diffeomorphism $g$ sufficiently $C^1$-close to $f$ has an ergodic measure
$\mu_g$  close to $\mu$ in the weak$\ast$ topology and with entropy close to the one of $\mu$. If the measure is hyperbolic this is essentially a reformulation of Remark~\ref{rem:Katok}. In the nonhyperbolic case, just note
that the measures supported on $\Gamma^\pm$ are close (in the weak$\ast$ and in entropy) to $\mu$. Hence
measures supported on the (well  and uniquely defined) continuations of $\Gamma^\pm$ for diffeomorphisms nearby $f$ are close to $\mu$. Note that these 
continuations of $\mu$ are hyperbolic. A much more interesting question, related to Theorem~\ref{theo:main3twin}, is if for $g$ close to $f$ the diffeomorphism $g$ has a nonhyperbolic measure close to $\mu$ (in the weak$\ast$ and in entropy).  This remains an open question. Note that by~\cite{BocBonDia:16}, $C^1$-open and -densely,  the diffeomorphisms close to $f$ have nonhyperbolic ergodic measures with positive entropy, but it is unclear and unknown if those can be chosen close to $\mu$.

Finally, observe that our constructive method provides a  way to obtain the hyperbolic sets $\Gamma^\pm$ (and hence their continuations) based on skeletons, see Section~\ref{skeleleton}.
Our notion of skeleton follows the one introduced in \cite{DiaGelRam:17} and depends on a blender-horseshoe, two connection times to such a blender-horseshoe, and finitely many (long) finite segments of orbits (where the finite central exponent is close to zero). In our context, all these ingredients are persistent. 
 Our concept of \emph{skeleton}  is  different (although with somewhat similar flavor) from the one introduced in parallel in \cite{DolViaYan:16}, that we call here \emph{DVY-skeleton}. The latter is a finite collection of hyperbolic periodic points with no heteroclinic intersections such that the strong unstable leaf of any point $x$ in the manifold intersects transversally the stable manifold of the orbit of some point in the skeleton. Open and densely  in $\cORTPH^1(M)$, DVY-skeletons
 consist of just one point (this follows from the minimality of the strong foliations and by the fact that the manifold is a homoclinic class, see Section~\ref{sscor:simplleexx}). Note that, in general, the DVY-skeletons may collapse by perturbations.
\end{remark}

The space $\cM(f)$ equipped with the weak$\ast$ topology is a Choquet simplex whose extreme points are the ergodic measures (see~\cite[Chapter 6.2]{Wal:82}). In some cases the set of ergodic measures $\cM_{\rm erg}(f)$ is dense in its closed convex hull $\cM(f)$ in which case (assuming that $\cM(f)$ is not just a singleton) one refers to it as \emph{the Poulsen simplex}%
\footnote{{Given a nonempty metrizable convex compact subset $K$ of a locally convex topological vector space, we say that $K$ is a \emph{Choquet simplex} if every point of $K$ is the barycenter of a unique probability measure supported on the set of extreme points of $K$. A \emph{Poulsen simplex} is a Choquet simplex 
where the extreme points are dense in $K$. See \cite{Sim:11}. 
}}, see also~\cite{LinOlsSte:78}. Although, in general, $\cM(f)$ is very far from having such a property, it is a consequence of~\cite{BocBonGel:} that each of the subsets $\cM_{\rm erg,<0}(f)$ and $\cM_{\rm erg,>0}(f)$ is indeed a Poulsen simplex. We  investigate further these simplices and study the remaining set of nonhyperbolic (ergodic) measures. Properties of this flavour were also studied in \cite{AbdBonCro:11}. {Let us observe that it is still is an open question whether hypotheses (H1)--(H3) imply that $\cM(f)$ itself is a Poulsen simplex.} 

\begin{mtheorem}[Arcwise connectedness]\label{teocor:simplleexx}
There is an $C^1$-open and -dense subset {of $\cORTPH^1(M)$ } consisting of diffeomorphisms $f$ for which the  intersection of the closed convex hull of $\cM_{\rm erg,<0}(f)$  and the closed convex hull of $\cM_{\rm erg,>0}(f)$  is nonempty and contains $\cM_{\rm erg,0}(f)$.
Each of the sets $\cM_{\rm erg,<0}(f)$ and $\cM_{\rm erg,>0}(f)$ is arcwise connected. Moreover, every measure in $\cM_{\rm erg,0}(f)$ is arcwise connected with any measure in $\cM_{\rm erg,<0}(f)$ and $\cM_{\rm erg,>0}(f)$, respectively. 
\end{mtheorem}

Indeed, the open and dense subset in the above corollary is the subset of $\cORTPH^1(M)$ for which the entire manifold is simultaneously the homoclinic class of a saddle of index $s$ and of index $s+1$, respectively. See the proof of Theorem~\ref{teocor:simplleexx}.

 The above theorem partially extends results in \cite{GorPes:17}  to our $C^1$ partially hyperbolic setting. The results in~\cite{GorPes:17} are stated for (i) measures supported on an \emph{isolated} homoclinic class whose saddles of the same index are all homoclinically related and assuming that (ii) $f$ is $C^{1+\varepsilon}$. Concerning (ii),  nowadays it is often used that the hypothesis $C^{1+\varepsilon}$ can be replaced by $C^1$ plus domination. Concerning (i), we will see that these conditions are satisfied in our setting. Indeed, see Section~\ref{sscor:simplleexx},  the set  $\cORTPH^1(M)$ can be chosen such that these two hypotheses hold for every of its elements. Theorem~\ref{teocor:simplleexx} is proved in  Section~\ref{sscor:simplleexx}. See also \cite[Section 3.1]{DiaGelRam:17b} for a proof of this type of results in a step skew product setting.

\subsection{The axiomatic approach in \cite{DiaGelRam:17,DiaGelRam:17b}}\label{ss:axioo}
As we have mentioned, this paper is a continuation of \cite{DiaGelRam:17,DiaGelRam:17b}, where the corresponding results where obtained for step skew-products with  circle fiber maps. The axiomatic setting proposed in \cite{DiaGelRam:17} considers three main hypotheses formulated for the underlying iterated function system (IFS) of the skew-product: transitivity, controlled expanding (contracting) forward covering relative to an interval (called \emph{blending interval}), and forward (backward) accessibility relative to an interval. In \cite[Section 8.3]{DiaGelRam:17} it is explained how these conditions are in fact motivated by the setting of  diffeomorphisms in $\cORTPH^1(M)$:
 the controlled expanding (contracting) forward covering property mimics the existence of  expanding (contracting) blenders, while  the forward (backward) accessibility mimics the minimality of the strong unstable (stable) foliation.
 As discussed in \cite{DiaGelRam:17b}, the axioms mentioned above  capture the essential dynamical properties of diffeomorphisms in $\cORTPH^1(M)$. In this paper, we complete the study initiated in  \cite{DiaGelRam:17,DiaGelRam:17b}.
 A key ingredient in the study of $\cORTPH^1(M)$ is the minimality of the strong invariant foliations. In \cite{BonDiaUre:02} blender-horseshoes are used  to prove this minimality, although at that time this concept was not yet introduced and the term blender-horseshoe does not appear in \cite{BonDiaUre:02}, and the authors refer to so-called \emph{complete sections} (see Section~\ref{sec:occurr} and Proposition~\ref{p.blenderopendense}). The next step, once these blender-horseshoes are obtained, is to study their dynamics and to state the precise correspondence of their expanding/contracting covering properties. This is done here in Section~\ref{sec:blennn} and Proposition~\ref{p.c.coveringproperty}.

\subsection{Idea of the proof}\label{subsec:Idea}
The proof is essentially based on the following ingredients. First we use blender-horseshoes with are just hyperbolic
basic sets with an additional geometrical superposition property. The second ingredient are the minimal strong foliations. Our construction will use so-called skeletons. A skeleton $\mathfrak{X}$  consists of arbitrarily long orbit pieces that mimic the ergodic theoretical properties of the given nonhyperbolic measure $\mu$. 
The cardinality of the skeletons $\card \mathfrak{X}$  is of order of $e^{m h(\mu)}$, where $m$ is the length of each individual orbit segment in the skeleton.
Using minimality, we see that these segments can be  connected in uniformly bounded time to the ``domain of the blender". Technical difficulties are the control of distortion related to the central direction as well as the absence of a central foliation. This last difficulty is circumvented by the use of  ``fake local invariant foliations"
introduced in \cite{BurWil:10}.

The hyperbolic set in Theorem~\ref{t.approx} is obtained as follows: using the segments of orbits provided by the skeleton property we construct $\card \mathfrak{X}$ pairwise disjoint full rectangles in the ``domain of the blender" such that for a fixed iterate
$N$ (which is of the order of $m$) the image of each rectangle intersects in a Markovian way each rectangle. 
This provides a hyperbolic basic set whose entropy is close to $h(\mu)$ and its exponents are close to $0$.

\subsection*{Organization of the paper}
In Section~\ref{sec:blennn}, we review all ingredients to construct blender-horseshoes and state and prove Proposition~\ref{p.c.coveringproperty} about the controlled expanding/contracting forward central covering property. In that section, we also prove  their $C^1$-open and -dense occurrence in $\cRTPH^1(M)$. In Section~\ref{s.approximation}, we state a general result on how to approximate the individual quantifiers of an ergodic measure by individual orbits. In Section~\ref{s.fake}, we recall fake invariant foliations 
to deal with the problem that in general there is no foliation tangent to the central bundle.
 Section~\ref{sec:5THEPROOF} is dedicated to the proof of Theorem~\ref{t.approx} and is the core of this paper. Theorem~\ref{theo:main3twin} is proven in Section~\ref{sec:theo:main3twin}, while Section~\ref{sscor:simplleexx} gives the proof of Theorem~\ref{teocor:simplleexx}.

\section{Blender-horseshoes}\label{sec:blennn}

In this section, we review the construction of blender-horseshoes in \cite{BonDia:12}
using the existing partially hyperbolic structure of the  diffeomorphisms. Here, besides the topological properties of blender-horseshoes, we will also need an additional quantitative \emph{controlled expanding forward central covering}, see Proposition~\ref{p.c.coveringproperty}.
In Section~\ref{sec:occurr}, we state the open and dense occurrence of blender-horseshoes in our setting, see Proposition~\ref{p.blenderopendense}.

\subsection{Definition of a blender-horseshoe}\label{ss.defblender}
We  will follow closely the presentation of blender-horseshoes in \cite{BonDia:12}  based on ingredients such as hyperbolicity, cone fields, and Markov partitions, and sketch its main steps. We also provide some further information which is not explicitly stated in \cite{BonDia:12}.

We say that a maximal invariant set $\Lambda$ of $f$ is an {\emph{unstable blender-horseshoe}} if there exists a region $\mathbf{C}$ diffeomorphic to $[-1,1]^{s+1+u}$ such that
$$
	\Lambda
	\eqdef \bigcap_{i\in \mathbb{Z}} f^i(\mathbf{C}) 
	\subset \mathrm{int}(\mathbf{C})
$$
and $\Lambda$ is a hyperbolic set with $s$-dimensional stable bundle and $(1+u)$-dimensional unstable bundle which satisfies conditions (BH1)--(BH6) in \cite[Section 3.2]{BonDia:12}. The set $\mathbf{C}$ is the \emph{domain of the blender-horseshoe}. A \emph{stable blender-horseshoe} is an unstable blender-horseshoe for $f^{-1}$.
Roughly speaking, it is a ``horseshoe with two legs'' having specific properties and being embedded in the ambient space in a especial way that it is has a ``geometric superposition property'': stated in the simplest way, there is an interval $(a,b)\subset[-1,1]$ such that for every $(x^\s,x)\in[-1,1]^s\times(a,b)$ any disk of the form $D=\{(x^s,x)\}\times[-1,1]^u$ intersects the local stable manifold of $\Lambda$. A key feature is that this property also holds for perturbations of  such disks.

To explain the simplest model, consider an affine horseshoe map $f$ such that in the central direction the map acts as a multiplication $x\mapsto\lambda x$ for some $\lambda\in(1,2)$; the maximal compact invariant set being contained in the rectangle $[-1,1]^s\times\{0\}\times[-1,1]^u$. Refer to Figure~\ref{Fig:blender} and the notation there. Note that this rectangle is not normally hyperbolic but $f$ is partially hyperbolic (consider the case of $\varepsilon=0$ in Figure~\ref{Fig:blender}). We now perturb $f$ in such a way, keeping affinity, that ``one of the legs is moved to the left'' in the central direction changing the dynamics in the central direction in  the rectangle $\mathbf C_{\mathbb B}$ to $x\mapsto \lambda x-\varepsilon$, $\varepsilon>0$ small. This provides an example of an affine unstable blender-horseshoe where the domain is $\mathbf C=[-1,1]^s\times[-\delta,\varepsilon(\lambda-1)^{-1}+\delta]\times[-1,1]^u$, $\delta>0$ small. A precise construction with all the details can be found in \cite{BonDiaVia:95} (though the term blender is not used there). Indeed, this example corresponds to the prototypical blender-horseshoes in~\cite[Section 5.1]{BonDia:12}.
Figure~\ref{Fig:blender} shows  a prototypical blender-horseshoe and illustrates at the same time all the elements in the (general) construction in this section.

\begin{figure}[h] 
\begin{overpic}[scale=.5, %grid, tics=10
]{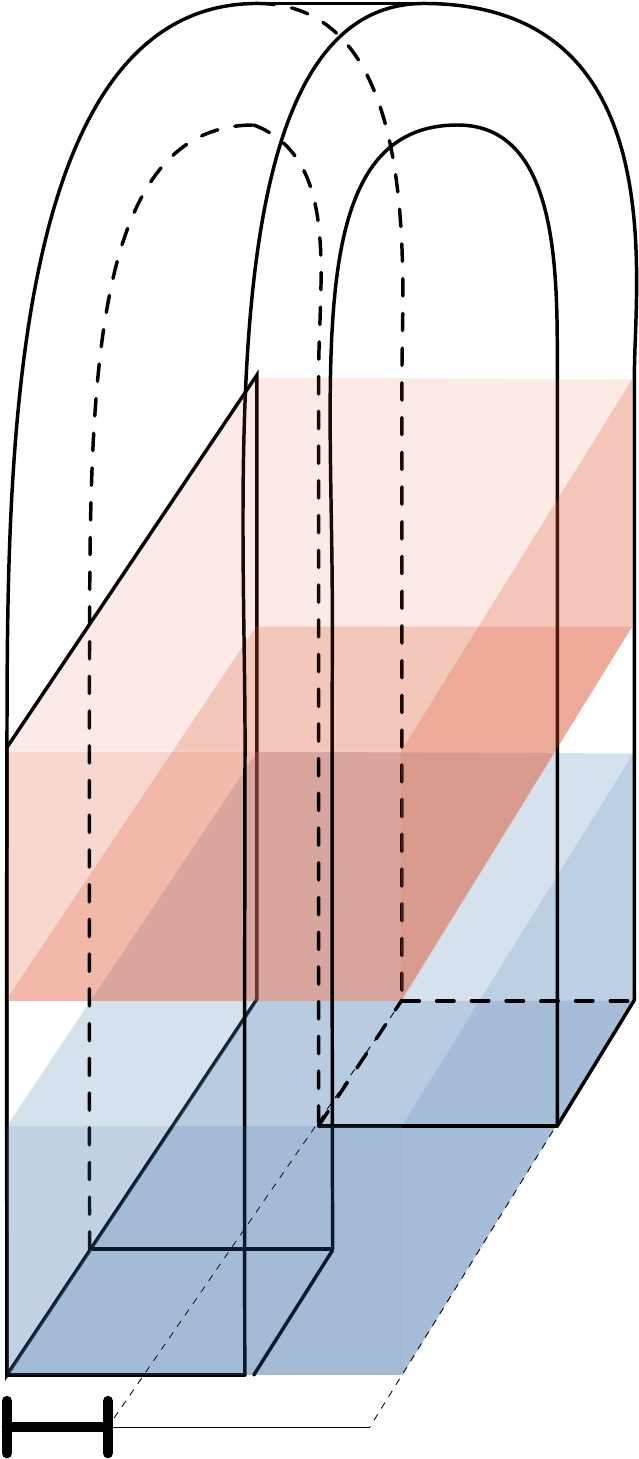}
 \put(3,-3){{\small{$\varepsilon$}}}
 \put(-8,38){{{$\mathbf C_{\mathbb B}$}}}
 \put(-8,12){{{$\mathbf C_{\mathbb A}$}}}
 \end{overpic}
 \hspace{2cm}
\begin{overpic}[scale=.3, %grid, tics=10
]{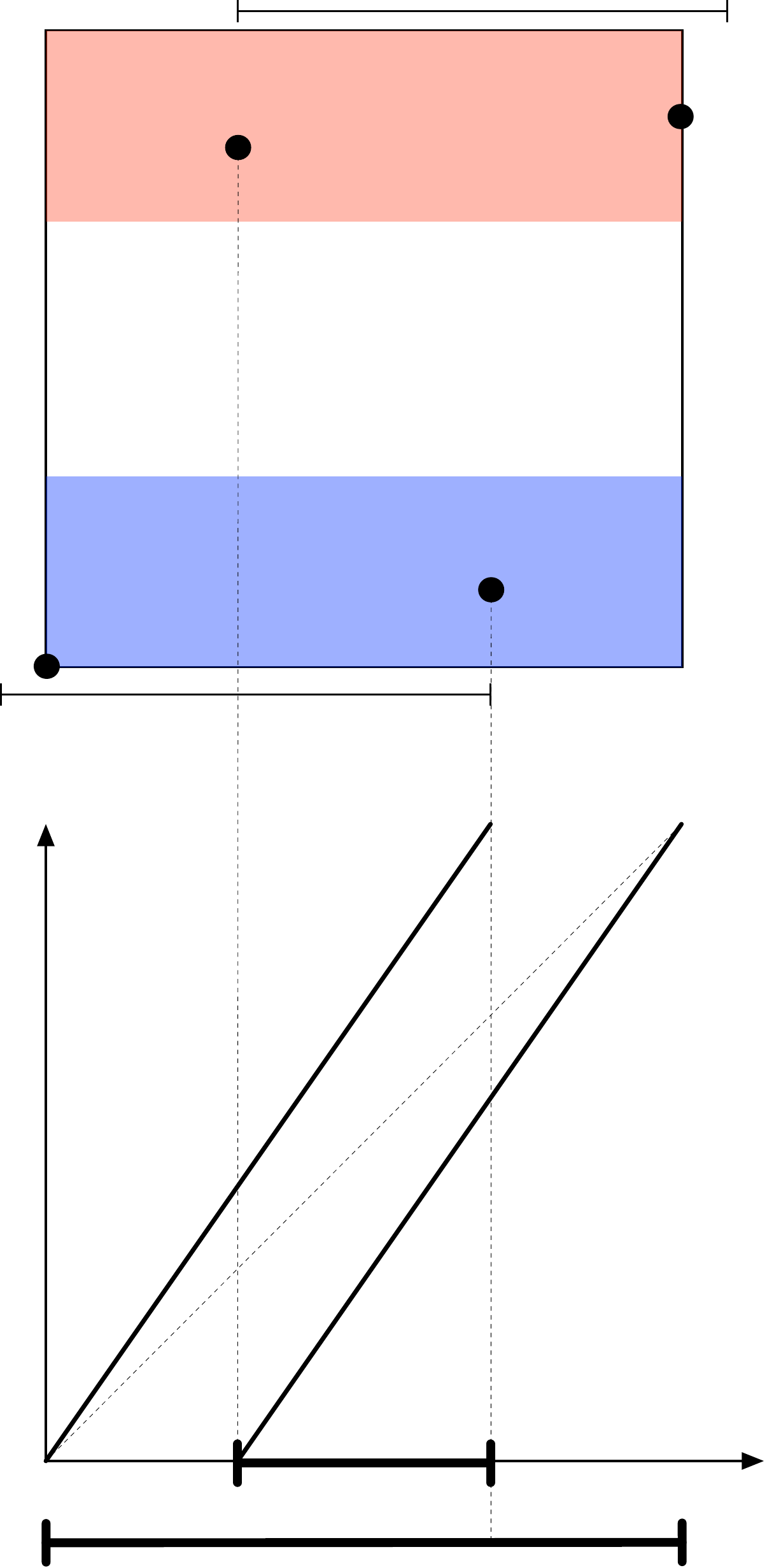}
\put(40,101){{\small{$\tau$}}}	
\put(7,52){{\small{$\tau$}}}	
 \put(-5,57){{\small{$P$}}}
 \put(44,93){{\small{$Q$}}}
 \put(6,91){{\small{$x_P$}}}
 \put(22,63){{\small{$x_Q$}}}
 \put(51,3){{\small{$x$}}}
 \put(15,3){{\small{$\varrho$}}}
 \put(0,-3){{\small{$\nu$}}}
 \end{overpic}
  \caption{Affine blender-horseshoe}
\label{Fig:blender}
\end{figure}
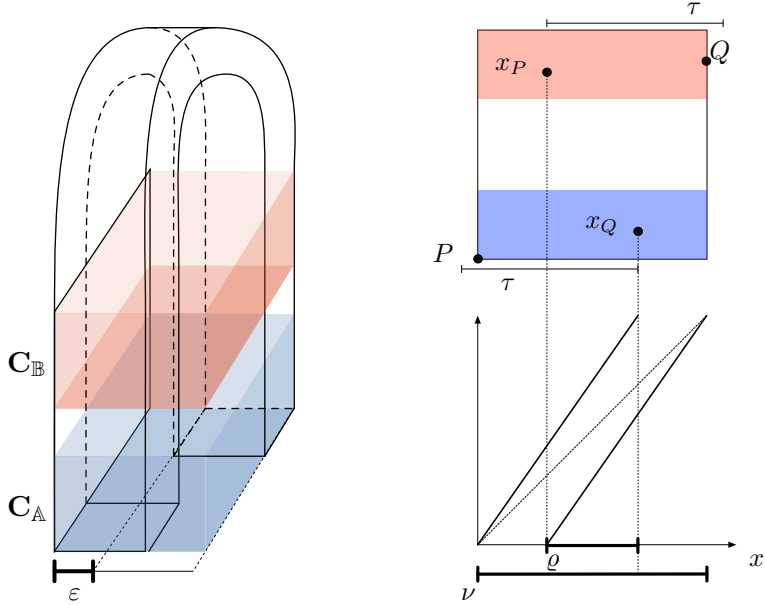

The main result in this section is Proposition \ref{p.c.coveringproperty} which derives a controlled  expanding forward central covering property, that is, the existence of some forward iteration along which any small enough unstable strip $S$ ``crossing the domain of the blender-horseshoe''  is uniformly expanded (in the central direction) and covers (in the central direction) the entire domain. This occurs  with uniform control on iteration length and expansion strength which depend on the central size of $S$ only. This property has its correspondence to the Axiom CEC+ in \cite{DiaGelRam:17} there stated for an IFS. 

Recall that we assume that $f$ is a partially hyperbolic diffeomorphism with a globally defined splitting $E^\ss\oplus E^\c \oplus E^\uu$, where  $s=\dim E^\ss\ge 1$, $u=\dim E^\uu\ge 1$, and $\dim E^\c=1$. 
Here the hyperbolic structure of the blender-horseshoe fits nicely with the partially hyperbolic one of $f$. In particular, $E^\s=E^\ss$ and $E^\u=E^\c\oplus E^\uu$ and the stable manifolds are the strong stable manifolds of $f$.

Conditions (BH1) and (BH3) in \cite{BonDia:12} state  the existence of a Markov partition and, in particular,  imply that the set $\Lambda$ is conjugate to a full shift of two symbols, denoted by $\mathbb{A}$ and $\mathbb{B}$.
The Markov partition provides two disjoint ``sub-rectangles" $\mathbf{C}_{\mathbb{A}}$  and  $\mathbf{C}_{\mathbb{B}}$  of $\mathbf{C}$ that codifies the dynamics, that is, $\Lambda= \bigcap_{i\in \mathbb{Z}} f^i(\mathbf{C}_{\mathbb{A}} \cup  \mathbf{C}_{\mathbb{B}} )$ and to each point $x\in \Lambda$ the conjugation associates the sequence $(\xi_i)_{\in \mathbb{Z}}\in 
\{\mathbb{A},\mathbb{B}\}^\mathbb{Z}$ defined by $f^i(x)\in \mathbf{C}_{\xi_i}$. This implies that $f$ has a fixed point $P\in \mathbf{C}_{\mathbb{A}}$ and a fixed point $Q\in \mathbf{C}_{\mathbb{B}}$.

Condition (BH2) refers to the existence of strong stable $\cC^\ss$, strong unstable $\cC^\uu$, and unstable $\cC^\u$ invariant cone fields (about  $E^\ss$, $E^\uu$, and $E^\u\eqdef E^\c\oplus E^\uu$, respectively). More precisely, given $\vartheta>0$ we denote 
\[\begin{split}
	\cC^\ss_\vartheta
	&\eqdef\{v=v^\ss+v^\c+v^\uu\colon v^i\in E^i,i\in\{\ss,\c,\uu\},
		\lVert v^\c+v^\uu\rVert\le\vartheta\lVert v^\ss\rVert\}.
\end{split}\]
 We simply refer to $\cC^\ss$ if $\vartheta$ is not specified. Analogously for $\cC^\u,\cC^\uu$. Here we also consider a cone field $\cC^\c$ contained in $\cC^\u$ about the central bundle with the analogous definition. 
Note that $\cC^\ss$ is backward invariant, $\cC^\u$ and $\cC^\uu$ are forward invariant, while $\cC^\c$ is not invariant. In our case, due to the partial hyperbolicity, the (global) existence of these cone fields is automatic and the  key point is the existence of $\lambda_{\rm bh} >1$ (and some appropriate norm $\lVert\cdot\rVert$ equivalent to the initial one, \cite{Gou:07}) such that
\begin{equation}\label{e.nexpansionCu}
	\lVert Df_x(v) \rVert
	\ge \lambda_{\rm bh} \lVert v\rVert, \quad \mbox{for every $x\in \mathbf{C}_{\mathbb{A}} \cup  \mathbf{C}_{\mathbb{B}}$ and $v\in \cC^\u$}.
\end{equation}
This means that the, otherwise neutral, central direction is indeed expanding in  $\mathbf{C}_{\mathbb{A}}$ and $\mathbf{C}_{\mathbb{B}}$.

The explanation of  the remaining conditions (BH4)--(BH6) demands some preliminary work. We consider the parts the boundary  of the ``rectangle" $\mathbf{C}$
corresponding to $(\partial [-1,1]^s) \times [-1,1] \times [-1,1]^u$  and
$[-1,1]^s \times [-1,1] \times ( \partial [-1,1]^u)$ and call them \emph{strong stable} and \emph{strong unstable boundaries}, denoted by $\partial^\ss \mathbf{C}$ and  $\partial^\uu \mathbf{C}$, respectively.%
\footnote{Note that in \cite{BonDia:12}, $\partial^\ss\mathbf C$ is called \emph{stable boundary} and denoted by $\partial^\s\mathbf C$. As here simultaneously we have stable and strong stable bundles, we prefer this notation.}

\smallskip\noindent\textbf{$\ss$-complete and $\uu$-complete disks.}
A {\emph{$\ss$-complete disk}} is a disk of dimension $s$ (that is a set diffeomorphic to $[-1,1]^s$) contained in $\mathbf{C}$ and tangent to the cone field $\cC^\ss$ whose boundary is contained in $\partial^\ss \mathbf{C}$. Similarly, a \emph{$\uu$-complete disk} is a disk of dimension $u$ contained in $\mathbf{C}$ and tangent to the cone field $\cC^\uu$ whose boundary is contained in $\partial^\uu \mathbf{C}$. It turns out that $\ss$- and $\uu$-complete disks containing a point $x\in \mathbf{C}$ are not unique. 
\smallskip

The \emph{local stable  manifold} $\cW^\st_\loc (x,f)$ of a point $x\in \Lambda$ is the connected component of $\cW^\st(x,f)\cap \mathbf{C}$ that contains $x$.%
\footnote{Note that here $\cW^\st_\loc(x,f)=\cW^\ss_\loc(x,f)$.}
 Similarly, for the local strong unstable manifold $\cW^\uu_\loc (x,f)$ of $x\in \Lambda$. Note that $\cW^\st_\loc (x,f)$ is a $\ss$-complete disk and  $\cW^\uu_\loc (x,f)$ is a $\uu$-complete disk for every $x\in \Lambda$.

 Condition (BH4) is a geometrical condition that claims that $\uu$-complete disks 
cannot intersect simultaneously $\cW^\st_\loc (P,f)$ and $\cW^\st_\loc (Q,f)$. 

\smallskip\noindent\textbf{$\uu$-complete disk in-between.}
Condition (BH4) also implies there are two homotopy classes of $\uu$-complete disks in $\mathbf{C}$ disjoint from $\cW^\st_{\loc}(P,f)$, called \emph{disks to the right} and \emph{disks to the left of $\cW^\st_{\loc}(P,f)$}. Similarly for $\cW^\st_{\loc}(Q,f)$. 
A $\uu$-complete disk that is to the right of $\cW^\st_{\loc}(P,f)$ and to the left of $\cW^\st_{\loc}(Q,f)$ is called \emph{in-between $\cW^\st_{\loc}(P,f)$ and $W^\st_{\loc}(Q,f),$} or shortly \emph{in-between}. We denote these disks by $\mathcal{D}^\uu_{\mathrm{bet}}$. Choosing appropriately right and left, we have
$\cW^\uu_{\loc}(x,f)\in \mathcal{D}^\uu_{\mathrm{bet}}$  for every $x\in \Lambda\setminus \{P,Q\}$.
\smallskip

 For each $\uu$-complete disk in  $D^\uu$ we consider the sets
 $$
 	D^\uu_{\mathbb{A}}
	\eqdef D^\uu  \cap \mathbf{C}_{\mathbb{A}}
 	\quad \mbox{and}\quad 
  	D^\uu_{\mathbb{B}}
	\eqdef D^\uu \cap \mathbf{C}_{\mathbb{B}}.
$$
Conditions (BH5)--(BH6) claims that for every $D^\uu \in \mathcal{D}^\uu_{\mathrm{bet}}$ then either $f ( D^\uu_{\mathbb{A}})\in \mathcal{D}^\uu_{\mathrm{bet}}$ or   $f ( D^\uu_{\mathbb{B}})\in \mathcal{D}^\uu_{\mathrm{bet}}$ (and there are cases such that both sets are in-between). This concludes the sketch of the description of a blender-horseshoe.

\begin{remark}[Orientation]\label{rem:orientation}
	Recall again that there is a (global) partially hyperbolic splitting of the tangent bundle of the manifold $TM=E^\ss\oplus E^\c\oplus E^\u$. In the definition of a blender-horseshoe, we will also require that for $f$ restricted to $\mathbf C_{\mathbb A}\cup\mathbf C_{\mathbb B}$, the tangent map $Df$ preserves orientation in the bundle $E^\c$. Note that this is also implicitly assumed in \cite{BonDia:12}.
\end{remark} 
 
\subsection{$\u$-strips in-between and expanding central covering}
 
 Similarly as in \cite[Section 1.a]{BonDia:96}, we introduce the notion of a \emph{$\u$-strip}. First, a curve in $\mathbf{C}$ is called \emph{central} if it is tangent to $\cC^\c$. A  \emph{$\u$-strip} is a closed disk $S$ of dimension $1+u$ tangent to the unstable cone field $\cC^\u$ that is simultaneously foliated by $\uu$-complete disks and by central curves (a \emph{central foliation} of $S$).
 Given a $\u$-strip $S$, a curve $\alpha\subset S$ is called \emph{$(S, \mathrm{c})$-complete} if it is a curve (whole leaf%
 \footnote{We define the strong unstable boundary of a strip in the same spirit of $\partial^\uu\mathbf C$, a complete leaf joins the two components of that boundary.}) of some central foliation of $S$. To a $\u$-strip $S$ we associate its (\emph{inner}) \emph{width} defined by
\[
 	w(S)  
	\eqdef \inf\{ |\alpha| \colon  \mbox{$\alpha$ is $(S, \mathrm{c})$-complete}\}.
\]
We say that a $\u$-strip is {\emph{in-between}} if it is foliated by $\uu$-complete disks in-between. To each $\u$-strip $S$ in-between we associate the sets
 $S_{{\mathbb{A}}}\eqdef S \cap  \mathbf{C}_{\mathbb{A}}$ and 
  $S_{{\mathbb{B}}} \eqdef S \cap  \mathbf{C}_{\mathbb{B}}$. 
We say that a $\u$-strip $S$ is \emph{$\c$-complete} if its intersects simultaneously $\cW^\st_{\loc}(P,f)$ and $\cW^\st_{\loc}(Q,f)$.

\begin{remark} \label{r.npossibilitiesstrip}
Conditions (BH5) and (BH6) and the expanding condition \eqref{e.nexpansionCu} imply that for a given $\u$-strip $S$ which is in-between there are two possibilities (see the arguments in \cite[Lemma 4.5]{BonDia:12}): 
\begin{enumerate}
\item	either $f(S_{\mathbb{A}})$ or $f(S_{\mathbb{B}})$ contains a $\u$-strip $S'$ in-between with $w(S')\ge \lambda_{\rm bh} \, w (S)$,
\item
or either $f(S_{\mathbb{A}}) \cap  \cW^\st_{\loc}(P,f)\ne\emptyset$ or
$f(S_{\mathbb{B}}) \cap  \cW^\ut_{\loc}(P,f)\ne\emptyset$.
\end{enumerate}
Moreover, if $S$ is $\c$-complete then $f(S_{\mathbb{A}})$ and $f(S_{\mathbb{B}})$ are both $\c$-complete. 
\end{remark}

For our goals we need a more precise ``quantitative"  version of the ``expanding" returns in the remark, that we call \emph{controlled expanding forward central covering} stated below.

\begin{proposition}[Controlled expanding forward central covering]\label{p.c.coveringproperty}
Let $\lambda_{\rm bh}>1$ be as in \eqref{e.nexpansionCu}. There is $C>0$ such that for every $\u$-strip $S$ in-between there is a positive integer $\ell (S)$, 
$$
	\ell (S)
	\eqdef \left\lceil\frac{\lvert\log w (S)\rvert}{\log \lambda_{\rm bh}} + C\right\rceil+1,
$$
such that for every $\ell \ge \ell(S)$ there
is a subset $ S'\subset S$ such that
\begin{itemize}
\item[(a)] $f^k( S')$ is  contained in $\mathbf{C}$ for all $k\in\{0,\ldots, \ell\}$ and 
\item[(b)] $f^\ell( S')$ is a $\c$-complete $\u$-strip.
\end{itemize}
\end{proposition}

The proof of the above proposition will be completed in Section~\ref{proof:Prooo}.

\subsection{Further properties of blender-horseshoes}\label{sss.additional}
To get  Proposition~\ref{p.c.coveringproperty}, we state additional properties (BH7), (BH8), and (BH9). Note that they are not additional hypotheses on the blender-horseshoe but rather  straightforward consequences of (BH1)--(BH6) and the constructions in \cite{BonDia:12} and obtained taking a sufficiently thin strong unstable cone field. 
\begin{enumerate}
\item[(BH7)] 
The intersection $f^{-1} (\cW^\st_{\loc} (P,f))\cap \mathbf{C}$  consists of two connected components:
$\cW^\st_{\loc} (P,f)$ and a second component $\cW^\st_{\loc} (x_P,f)$, where $x_P$ is a homoclinic point%
\footnote{A point is a \emph{homoclinic point of $P$} if it belongs simultaneously to the stable and to the unstable manifold of $P$. Note that, in our setting, a homoclinic point is automatically \emph{transverse}, that is, those manifolds intersect transversally.}
 of $P$ in $\Lambda$. Similarly, $f^{-1} (\cW^\st_{\loc} (Q,f))\cap \mathbf{C}$ consists of  two connected components, $\cW^\st_{\loc} (Q,f)$ and $\cW^\st_{\loc} (x_Q,f)$, where $x_Q$ is a homoclinic point of $P$ in $\Lambda$. 
\end{enumerate}

As above, we can speak of $\uu$-complete disks to the left/right of  $\cW^\st_{\loc} (x_P,f)$ and of $\cW^\st_{\loc} (x_Q,f)$.
Similarly as in condition (BH4) the blender-horseshoe we have the following:
\begin{enumerate}
\item[(BH8)]
Every $\uu$-complete disk which intersects $\cW^\st_{\loc} (x_P,f)$ is to the left of 
$\cW^\st_{\loc} (x_Q,f)$ and every $\uu$-complete disk intersecting $\cW^\st_{\loc} (x_Q,f)$ is to the right of $\cW^\st_{\loc} (x_P,f)$. In particular,  
any $\uu$-complete disk intersecting $\cW^\st_{\loc} (x_P,f)$ and any 
$\uu$-complete disk intersecting $\cW^\st_{\loc} (x_Q)$ are disjoint. Moreover, there is $\varrho>0$ so that every $\u$-strip intersecting  $\cW^\st_{\loc} (x_P,f)$ and  $\cW^\st_{\loc} (x_Q,f)$ has minimal width bigger than $\varrho$.
\end{enumerate}

The points $x_P$ and $x_Q$ are auxiliary in order to quantify the size of the geometric superposition region (compare Figure~\ref{Fig:blender}). Note that, in order to prove Proposition \ref{p.c.coveringproperty} it is enough to show that given any $\u$-strip $S$ in-between there is  a number $n$ of iterates (depending on $\lvert\log w(S)\rvert$ only) so that we obtain a $\u$-strip which intersects the local stable manifold of $P$ and whose part to the right of $P$  has some least size (indeed, $\lambda_{\rm bh}\varrho$). This in turn is guaranteed when  $f^{n-1}(S)$ intersects simultaneously $\cW^\st_{\loc} (x_P,f)$ and  $\cW^\st_{\loc} (x_Q,f)$. This is a sketch of the content of Lemmas~\ref{l.intersectsimult},~\ref{l.both}, and~\ref{l.fullys}.

\begin{remark}\label{r.threecases}
Condition (BH8) implies that there are three possibilities for a $\u$-strip $S$  in-between:  
\begin{enumerate}
\item it is to  the left of $\cW^\st_{\loc}(x_Q,f)$,
\item  it is to the right of $\cW^\st_{\loc}(x_P,f)$,
\item it intersects simultaneously  $\cW^\st_{\loc}(x_P,f)$ and $\cW^\st_{\loc}(x_Q,f)$, hence by (BH8) it has minimal width at least $\varrho$.
\end{enumerate}
\end{remark}

Next condition is an improved version of Remark~\ref{r.npossibilitiesstrip} (and it is shown as in \cite[Lemma 4.5]{BonDia:12}).
\begin{enumerate}
\item[(BH9)] Consider a strip $S$ in-between. Then 
 \begin{enumerate}
 \item If $S$ is to the left of $\cW^\st_{\loc}(x_Q,f)$ then $f (S_{\mathbb{A}})$ contains a  $\u$-strip $S'$ in-between with $w(S')\ge \lambda_{\rm bh} w(S)$,
\item  If $S$ is to the right of $\cW^\st_{\loc}(x_P,f)$ then $f (S_{\mathbb{B}})$ contains a $\u$-strip $S'$ in-between  with $w(S')\ge \lambda_{\rm bh} w (S)$,
\end{enumerate}
\end{enumerate}

\begin{remark}\label{r.maximalwidth}
There is a number $\tau>0$ with the following property:
\begin{itemize}
\item Every $\u$-strip in-between to the  right of $\cW^\st_{\loc}(x_P,f)$ has (inner) width less than $\tau$. 
\item Every $\u$-strip in-between to the left of $\cW^\st_{\loc}(x_Q,f)$ has  (inner) width less than $\tau$. 
\end{itemize}
In other words, any $\u$-strip in-between with (inner) width bigger than $\tau$ intersects simultaneously $\cW^\st_{\loc}(x_P,f)$  and $\cW^\st_{\loc}(x_Q,f)$. Compare Figure~\ref{Fig:blender}.
\end{remark}

\subsection{Iterations of $\u$-strips} \label{sss.iterations}

The next step is the iteration of $\u$-strips to obtain covering properties. The key in this process is that here we have more accurate control of the image of the strips as in the (standard) blenders (compare with \cite[Lemma 1.7]{BonDia:96}).

\begin{lemma}[Simultaneous intersections]
\label{l.intersectsimult}
Consider a $\u$-strip $S$ in-between. Let $w=w (S)$ and define $N=N(w)$ as the first integer with $\lambda_{\rm bh}^{N} w>\tau$, where $\lambda_{\rm bh}$ is the expansion constant in \eqref{e.nexpansionCu} and $\tau$ is as in Remark~\ref{r.maximalwidth}.
Then there is a first $n\in\{0,\ldots, N\}$ such that 
\begin{itemize}
\item $f^n(S)$ contains a $\u$-strip $S'$ in-between that intersects simultaneously 
$\cW^\st_{\loc} (x_P,f)$ and $\cW^\st_{\loc} (x_Q,f)$,
\item we have $f^i(f^{-n}(S'))\subset \mathbf{C}$ for all $i=0,\dots,n$.
\end{itemize}
\end{lemma}

Considering the strip $S'$ in Lemma~\ref{l.intersectsimult} and recalling that $x_P$ is a homoclinic point of $P$ and $x_Q$ is a homoclinic point of $Q$, we have that $f(S'_{\mathbb{A}})$ intersects $\cW^\st_{\loc} (Q,f)$  and that $f(S'_{\mathbb{B}})$  intersects $\cW^\st_{\loc} (P,f)$.

\begin{proof}[Proof of Lemma~\ref{l.intersectsimult}]
The proof is by induction, using arguments as in \cite[Lemma 4.5]{BonDia:12}. Let $S^0=S$. If $S^0$ intersects simultaneously $\cW^\st_{\loc} (x_P,f)$ and $\cW^\st_{\loc} (x_Q,f)$ we are done. Otherwise, by Remark~\ref{r.threecases}, either
$S^0$ is to  the left of $\cW^\st_{\loc}(x_Q,f)$ or 
$S^0$ is to  the right of $\cW^\st_{\loc}(x_P,f)$.
If the first case consider $S^0_{\mathbb{A}}$ and observe that by (BH9) we have that $f (S^0_{\mathbb{A}})$ contains a $\u$-strip $S^1$ in-between with
$w(S^1)\ge \lambda_{\rm bh} w (S^0)$. In the second case,
consider $S^0_{\mathbb{B}}$ and observe that by (BH9) we have
that $f (S^0_{\mathbb{B}})$ contains a $\u$-strip $S^1$ in-between with
 $w(S^1)\ge \lambda_{\rm bh} w (S)$. Note that $f^{-1} (S^1) \subset S^0 \subset \mathbf{C}$.

We now proceed inductively, assume that we have defined $\u$-strips in-between $S=S^0,\dots, S^{n}$ that do not intersect simultaneously  $\cW^\st_{\loc}(x_P,f)$ and $\cW^\st_{\loc}(x_Q,f)$, satisfy either $S^{i}\subset f(S^{i-1}_{\mathbb{A}})$ or  $S^{i}\subset f(S^{i-1}_{\mathbb{B}})$, according to the case, and $w(S^i)\ge \lambda_{\rm bh}^i w$. As in the first inductive step, we take $S^{n+1}\subset f (S^n_{\mathbb{A}})$ if $S^n$ is to  the left of $\cW^\st_{\loc}(x_Q,f)$ or $S^{n+1}\subset f (S^n_{\mathbb{B}})$ otherwise. In both cases, we have that
$$
	w(S^{n+1})
	\ge \lambda_{\rm bh} w (S^n)
	\ge \lambda_{\rm bh}^{n+1} w.
$$
The choices of $\tau$ and $N$ imply that there is a first $n$ with $0\le n\le N$ such that $S^n$  intersects 
simultaneously 
$\cW^\st_{\loc} (x_P,f)$ and
$\cW^\st_{\loc} (x_Q,f)$. Hence  $f^n(S)$ 
contains a $\u$-strip  $S^\prime=S^n$ in-between that intersects simultaneously 
$\cW^\st_{\loc} (x_P)$ and
$\cW^\st_{\loc} (x_Q)$. Note that by construction 
$f^i(f^{-n}(S^n))$ is contained in $\mathbf{C}$ for all $i=0, \dots, n$. This completes the proof of the lemma.
\end{proof}

In what follows, for convenience, we consider a $\u$-strip $S$ together with a family of $\uu$-complete disks $\mathcal{D}_S=\{D^\uu_{S,i}\}_{i\in I}$ foliating $S$ (note that this foliation is not unique) and write $(S,\mathcal{D}_S)$. We say that a $\u$-strip $(S,\mathcal{D}_S)$ is {\emph{quasi to the right of  $\cW^\s_{\loc}(P,f)$}} if 
there is $i_0\in I$ with $D^\uu_{S,i_0} \cap \cW^\s_{\loc}(P,f)\ne\emptyset$
and for every $i\ne i_0$ the disk $D^\uu_{S,i}$ is to the right of $\cW^\s_{\loc}(P,f)$.
Note that this means, in particular, that the intersection with $\cW^\s_\loc(P,f)$ occurs in the strong unstable boundary of the strip.

\begin{lemma}\label{l.both}
Let $S$ be a $\u$-strip in-between which intersects simultaneously $\cW^\s_{\loc}(x_P,f)$ and $\cW^\s_{\loc}(x_Q,f)$.  Then $f(S_{\mathbb{B}})$ contains a $\u$-strip quasi to the right of $\cW^\s_{\loc}(P,f)$ with minimal width $\lambda_{\rm bh}\varrho$, where $\varrho$ was defined in (BH8).
\end{lemma}

\begin{proof}
The lemma follows recalling that $\u$-strips intersecting simultaneously  $\cW^\st_{\loc}(x_P,f)$ and $\cW^\st_{\loc}(x_Q,f)$ have minimal width at least $\varrho$ (see (BH8)) and using the expansion of central curves given by \eqref{e.nexpansionCu}. Recall also Remark~\ref{rem:orientation} about the preservation of the orientation. 
\end{proof}

Given a $\u$-strip $(S,\mathcal{D}_S)$ whose interior intersects $\cW^\s_{\loc}(P,f)$,  we consider the $\uu$-complete disk $D^\uu_{S,j}$ of $S$ intersecting $\cW^\s_{\loc}(P,f)$. 
Note that, since the intersection of $\cW^\s_{\loc}(P,f)$ with $S$ is transverse, the disk  $D^\uu_{S,j}$ is uniquely defined.
Observe  that $(S\setminus D^\uu_{S,j})$ has two connected components,
a component consisting of $\uu$-complete disks to the right of $\cW^\s_{\loc}(P,f)$ and a component consisting of $\uu$-complete disks to the left of $\cW^\s_{\loc}(P,f)$. We denote the closures of these components by $S^{\mathrm{right}}$ and $S^{\mathrm{left}}$ and observe that they intersect along the disk $D^\uu_{S,j}$. Note that $S^{\mathrm{right}}$ is quasi to the right of  $\cW^\s_{\loc}(P,f)$.
We can argue similarly with strips $S$ which are quasi to the right of  $\cW^\s_{\loc}(P,f)$, in that case $S=S^{\mathrm{right}}$ (thus $S^{\mathrm{left}}=\emptyset$).

Finally, note that there is a number $\nu>0$ such that every $\u$-strip $S$ that is
quasi to the right of  $\cW^\s_{\loc}(P,f)$ with $w(S)>\nu$ also intersects 
$\cW^\st_{\loc}(Q,f)$.

\begin{lemma}\label{l.fullys}
Consider a $\u$-strip $S$ with $S\cap \cW^\s_{\loc}(P,f)\ne \emptyset$ such that $w (S^{\mathrm{right}})=w>0$.
Define $L=L(w)$ as the first integer with $\lambda_{\rm bh}^{L}w>\nu$,
where $\lambda_{\rm bh}$ is the expansion constant in \eqref{e.nexpansionCu}.
Then for every $\ell\ge L$ it holds that $f^\ell(S^{\mathrm{right}})$  contains a $\c$-complete $\u$-strip $S^\ell$ such that $f^i(f^{-\ell}(S^\ell))\subset \mathbf{C}$ for all $i=0,\dots, \ell$.
 \end{lemma}

\begin{proof}
The proof follows as in Lemma~\ref{l.intersectsimult}. Let $S^0= S^{\mathrm{right}}$ and note that $f(S^{0}_{\mathbb{A}})$ contains $\u$-strip $S^1$ that is quasi to the right of $\cW^\s_{\loc}(P,f)$ and satisfies $w (S^1)\ge \lambda_{\rm bh} w (S^{\mathrm{right}})$. Now it is enough to argue inductively.
\end{proof}

\subsection{Proof of Proposition~\ref{p.c.coveringproperty}}\label{proof:Prooo}
Consider a $\u$-strip $S$ in-between and let $w=w (S)$.
By Lemma~\ref{l.intersectsimult}, there is a first 
$$
	0\le n 
	\le N 
	\le \max\Big\{0,\frac{ \log \tau/w}{\log\lambda_{\rm bh}}\Big\}+1
	\le \frac{\lvert\log w\rvert}{\log \lambda_{\rm bh}} +  C_1,
$$
for some $C_1$ independent of $w$,
such that $f^n(S)$ contains a $\u$-strip $S'$ in-between that intersects simultaneously $\cW^\st_{\loc} (x_P,f)$ and $\cW^\st_{\loc} (x_Q,f)$.
By Lemma~\ref{l.both}, we have that $f(S_{\mathbb{B}})$ contains a  $\u$-strip $\widetilde S$ quasi to the right of $\cW^\s_{\loc}(P,f)$ with $w (\widetilde S)\ge \lambda_{\rm bh}\varrho$. 
Note that $\widetilde S^{\mathrm{right}}=\widetilde S$. 
Take $L=L(\lambda_{\rm bh}\varrho)$ as in Lemma~\ref{l.fullys} and note that $f^L (\widetilde S)$ contains a $\c$-complete $\u$-strip $\widehat{S}$.
Note that, by the lemma, $\nu<\lambda_{\rm bh}^L(\lambda_{\rm bh}\varrho)\le\lambda_{\rm bh}\nu$ and hence $L\le C_2$ for some  universal constant $C_2$ independent on $w$. 
Recalling that the image of a $\c$-complete $\u$-strip contains a $\c$-complete $\u$-strip, see Remark~\ref{r.npossibilitiesstrip}, 
we have that  for every 
$k\ge n+1+L$, the set $f^k(S)$ contains a $\c$-complete $\u$-strip $\widehat S$ such that
its pre-image $f^{-k} (\widehat{S})\eqdef\widehat{S}_k \subset S$ satisfies 
(a) and (b) in the proposition.
 Finally, taking $\ell (S)=n+1+L$ we have
$$
	\ell(S)=n+1+L
	\le  \frac{\lvert\log w\rvert}{\log \lambda_{\rm bh}} +  C_1 +C_2,
$$
ending the proof of the proposition. \hfill \qed

\subsection{Occurrence of blender-horseshoes}\label{sec:occurr}
We close this section recalling the following result about the existence of blender-horseshoes.

\begin{proposition}\label{p.blenderopendense}
There is a $C^1$-open and -dense subset $\cORTPH^1(M)$  of the set $\cRTPH^1(M)$ consisting of diffeomorphisms $f$ such that there are an unstable blender-horseshoe for $f^{n_+}$ for some $n_+\ge1$ and an unstable blender-horseshoe for $f^{-n_-}$ for some $n_-\ge1$.
\end{proposition}

\begin{proof}
By \cite[Lemma 3.9]{BonDia:12},  having a blender-horseshoe is a $C^1$-open property. Let us now explain why having a blender-horseshoe is a $C^1$-dense property in $\cRTPH^1(M)$. In our context, due to the nonhyperbolicity assumption, we have that  $C^1$-open and -densely in $\cRTPH^1(M)$ the diffeomorphisms have simultaneously saddles of indices $\dim E^\ss$ and $\dim E^\ss +1$, this follows from the ergodic closing lemma in \cite{Man:82}. With the terminology in \cite{BonDia:08,BonDia:12}, the saddles of diffeomorphism in $\cRTPH^1(M)$  have real central eigenvalues (this follows from the fact that $\dim E^\c=1$). The robust transitivity assumption and the {\emph{connecting lemma}} \cite{Hay:97,BonCro:04}
 imply that $C^1$-densely in $\cRTPH^1(M)$ there are diffeomorphisms with heterodimensional cycles associated to these saddles with real central eigenvalues. By \cite[Theorem 3.3]{BonDia:08}, these cycles generate strong homoclinic intersections (saddle nodes whose strong stable and strong unstable manifolds meet quasi-transversally). Finally, \cite[Theorem 4.1]{BonDia:08} implies that by arbitrarily small $C^1$-perturbations these strong homoclinic intersections yield blender-horseshoes for some iterate of the map (stable or unstable, according to the chosen perturbation). We observe that though the terminology blender-horseshoe was not used in \cite{BonDia:08} the construction corresponds exactly to the prototypical blender-horseshoes in \cite[Section 5.1]{BonDia:12}. In this way, it follows we have shown that having (stable and unstable)  blender-horseshoes (for some iterate) is a $C^1$-dense  property in $\cRTPH^1(M)$.
\end{proof}

\begin{remark}[Choice of blender-horseshoes] \label{r.notation}
In what follows, we denote by  $\cORTPH^1(M)$ the $C^1$-open and -dense subset of $\cRTPH^1(M)$ of diffeomorphisms $f$ which have simultaneously an unstable  blender-horseshoe (for some iterates $f^{n_+}$) and an unstable blender-horseshoe (for $f^{-n_-}$). In what follows, for each $f\in\cORTPH^1(M)$ we fix an unstable blender-horseshoe 
$\Lambda^+$ with reference domain $\mathbf C^+$ with respect to some iterate $f^{n_+}$. For simplicity of notation, to emphasize the domain of the blender, we will write $(\Lambda_f^+,\mathbf C_f^+,f^{n_+})$ when referring to this blender  and we denote by  $P^+$ and $Q^+$ the corresponding fixed points (but  omitting the dependence of $n_+$, $P^+$, and $Q^+$ on $f$). 
\end{remark}

\subsection{Blender-horseshoes and strong foliations}
\label{ss.blenderfoliations}
Given an unstable blen\-der-horseshoe $(\Lambda_f^+,\mathbf C_f^+,f^{n_+})$ and a point $x\in \mathbf{C}_f^+$ denote by $\cF^{\uu}_{\mathbf{C}_f^+} (x)$ the connected component of $\cF^{\uu} (x) \cap \mathbf{C}_f^+$ containing $x$. We similarly define the set $\cF^\ss_{\mathbf C_f^+}(x)$. Above we  defined when a $\uu$-complete disk is \emph{in-between}.
Now, considering the sets
\[%\begin{equation}%\label{eq:walls}
	\Wall_f(R^+)
	\eqdef \bigcup_{x \in \cW^\st_{\loc}(R^+, f^{n_+})} \cF^\uu_{\mathbf{C}_f^+} (x)
	,\quad \text{ for }R\in\{P,Q\},
\]%\end{equation}
we say that a $\ss$-complete disk is \emph{in-between the $\uu$-walls} $\Wall_f(P^+)$ and $\Wall_f(Q^+)$ if it is disjoint with these two sets and intersects some $\uu$-complete disk in-between. The construction of a blender-horseshoe implies that there are $\ss$-complete disks in-between  the walls and that being in-between the walls is an open property.  

Given a set $U\subset \mathbf C^+_f$, we define its \emph{$\ss$-saturation} and its \emph{$\uu$-saturation} by
\[
	U^\ss
	\eqdef \bigcup_{x\in U}\cF^\ss_{\mathbf C_f^+}(x)
	\quad\text{ and }\quad
	U^\uu
	\eqdef \bigcup_{x\in U}\cF^\uu_{\mathbf C_f^+}(x).
\]
We say that $U$ is \emph{in-between the $\uu$-walls} of this blender-horseshoe if for every $x\in U$ the set $\cF^\ss_{\mathbf C_f^+}(x)$ is a $\ss$-complete disk  in-between the $\uu$-walls and  the set $\cF^\uu_{\mathbf C_f^+}(x)$ is a $\uu$-complete disk  in-between which is disjoint to the $\uu$-walls.

The fact that for every  $f\in \cORTPH^1(M)$ every strong unstable and every strong stable leaf is dense in the ambient space $M$, respectively, implies immediately the following lemma. Denote by $\cF^\uu(x,\delta)$ the ball centered at $x$ and with radius $\delta$ in the 
leaf $\cF^\uu(x)$ of the foliation $\cF^\uu$.  Define the set $\cF^\ss(x,\delta)$  similarly.

\begin{lemma}\label{l.minimalitysize}
Given $f \in \cORTPH^1(M)$ and an unstable blender-horse\-shoe $(\Lambda_f^+,\mathbf C_f^+,f^{n_+})$
consider open sets $V^+,U^+\subset\mathbf C_f^+$ which are in-between its $\uu$-walls such that $V^+\subset \overline{V^+}\subset U^+$.
\begin{itemize}
\item There is $\kappa_0=\kappa_0(V^+)>0$ such that every $x\in M$ and every $\kappa\ge \kappa_0$ the set $\cF^\uu(x,\kappa)$ contains a $\uu$-complete disk $\Delta^\uu\subset V^{+\uu}$ and the set $\cF^\ss(x,\kappa)$ contains a $\ss$-complete disk $\Delta^\ss\subset V^{+\ss}$.
\item  There is $\delta=\delta(U^+,V^+)>0$ such that for every $\uu$-complete disk $\Delta^\uu\subset V^{+\uu}$ and every $\u$-strip $S$ containing $\Delta^\uu$ of (inner) width $w(S)<\delta$ we have $S\subset U^{+\uu}$. In particular, $S$ is in-between.
\end{itemize}
\end{lemma}

Analogously to what was defined above, given any $x\in M$ and small $\delta>0$, denote by $\Delta^\uu(x,\delta)$ a $\uu$-disk centered at $x$ of radius $\delta$. Note that forward iterations of this disk by $f$ converges to segments of leaves of the strong unstable foliation $\cF^\uu$ while increasing exponentially its diameter. Analogously for $\ss$-disks $\Delta^\ss(x,\delta)$ and backward iterations. These observations lead to the following corollary of Lemma~\ref{l.minimalitysize}.

\begin{corollary}\label{cor:nofim}
In the setting of Lemma~\ref{l.minimalitysize} and with the same notation, 
for every $\delta>0$, there is $t_{\rm con}=t_{\rm con}(V^+,\delta)>0$ such that for every $x\in M$ and for every $t\ge t_{\rm con}$ the set $f^t(\Delta^\uu(x,\delta))$ contains a $\uu$-complete disk in $V^{+\uu}$ and the set $f^{-t}(\Delta^\ss(x,\delta))$ contains a $\ss$-complete disk in $V^{+\ss}$.

Note that $t_{\rm con}$ can be chosen such that
\[
	t_{\rm con}
	\le \frac{\log\,\lvert \kappa_0 (V^+) \rvert-\log\delta}
				{\log\lambda^\uu_{\rm min}} +C,
\]
where $\lambda^\uu_{\rm min}$ denotes the minimal expansion of $Df$ in the cone field $\cC^\uu$ and $C>0$ is some universal constant.
\end{corollary}

We will close this section by the following  ``safety" remark which we will use in Section~\ref{ss:prelim}

\begin{remark}[Safety neighborhoods]\label{rmlem:openneigh}
	Given $f \in \cORTPH^1(M)$, consider an unstable blender-horse\-shoe $(\Lambda_f^+,\mathbf C_f^+,f^{n_+})$ and open sets $V^+,N^+$, $U^+\subset\mathbf C_f^+$ which are in-between its $\uu$-walls and such that $V^+\subset \overline{V^+}\subset N^+\subset\overline{N^+}\subset U^+$. There is $\theta>0$  with the following property: Consider any $\u$-strip $S$ with $w(S)<\theta$ and containing a  $\uu$-complete disk in $V^{+\uu}$. Then $S\subset N^{+\uu}$ and $\cF^\ss(x,\theta)\cap\mathbf C_f^+\subset U^{+\uu}$ for every $x\in S$.
\end{remark}

\begin{remark}[Safety domain of a blender]\label{rem:safety}
	In the same spirit of the remark above, we consider a \emph{safety neighborhood of the domain of a blender-horseshoe}. Note that given a (say) unstable blender-horseshoe  $(\Lambda^+_f,\mathbf C^+_f,f)$ the set $\Lambda_f$ is contained in the interior of $\mathbf{C^+_f}$. We can assume that there is a slightly greater domain $\widehat{\mathbf C}^+_f$ (also homeomorphic to a rectangle) containing $\mathbf{C}^+_f$ in its interior where the cone fields can be extended (satisfying the same invariance and expansion/contraction properties) and such that the maximal invariant set of $\widehat{\mathbf C}^+_f$ is also $\Lambda_f$. We define the strong stable and strong unstable boundaries of $\widehat{\mathbf C}^+_f$ similarly as we did for $\mathbf{C}^+_f$ and note that corresponding boundaries are disjoint (and hence at some positive distance). Since the cone fields are defined on $\widehat{\mathbf C}^+_f$ we can speak of \emph{$\ss$-complete} and \emph{$\uu$-complete} disks and \emph{$\u$-strips relative to $\widehat{\mathbf C}^+_f$} (we will emphasize such a  dependence).
Any of such $\ss$-disk complete relative to $\widehat{C}_f^+$ contains a $\ss$- disk relative to $\widehat{\mathbf C}_f^+$. Similarly, for $\uu$-complete disks and 
$\u$-strips. We can also define in the obvious way the sets $\cF^{\uu}_{\widehat{\mathbf{C}}_f^+} (x)$ and
$\cF^{\ss}_{\widehat{\mathbf{C}}_f^+} (x)$ and the saturations
$V^{+\uu}_{\widehat{\mathbf{C}}^+_f}$ and $V^{+\ss}_{\widehat{\mathbf{C}}^+_f}$
 of a subset $V^+$ of $\mathbf{C}^+_f$.
\end{remark}

The next result is a straightforward extension of Lemma \ref{l.minimalitysize} and Corollary~\ref{cor:nofim} where a safety constant $\safe$ is introduced.

\begin{remark}\label{reml.minimalitysizesafety}
Given $f \in \cORTPH^1(M)$ consider an unstable blender-horse\-shoe $(\Lambda_f^+,\mathbf C_f^+,f^{n_+})$, open sets $V^+,N^+\subset\mathbf{C}_f^+$, with $V^+\subset \overline{V^+}\subset N^+$, which are in-between the $\uu$-wall of the blender, and a safety domain $\widehat{\mathbf C}^+_f$.
Then there is a \emph{safety constant} $\safe=\safe(\mathbf C^+_f,\widehat{\mathbf C}^+_f,V^+,N^+)>0$ such that if $\Delta^\uu$ is a $\uu$-disk  which is complete 
relative to $\widehat{\mathbf C}^+_f$ and contained in $V^{+\uu}_{\widehat{\mathbf C}^+_f}$ then every $\uu$-disk at Hausdorff distance less than $\safe$ with $\Delta^\uu$ contains a $\uu$-disk which is complete relative to $\mathbf{C}^+_f$ and contained in $N^{+\uu}$.

Moreover, the number $t_{\rm con}$ in Corollary \ref{cor:nofim} can be chosen such that for every $x\in M$ and for every $t \ge t_{\rm con}$ the set $f^t(\Delta^\uu(x,\delta))$ contains a $\uu$-disk in $V^{+\uu}_{\widehat{\mathbf C}^+_f}$ that is complete relative to $\widehat{\mathbf C}^+_f$.

Similarly, the number $t_{\rm con}$ can be chosen such that for every $x\in M$ and for every $t \ge t_{\rm con}$ the set $f^t(\cF^\ss(x,\delta))$ contains a $\ss$-disk in $V^{+\ss}_{\widehat{\mathbf C}^+_f}$ that is complete relative to $\widehat{\mathbf C}^+_f$.
\end{remark}

\section{Approximation of ergodic measures}\label{s.approximation}

The following is just a reformulation of \cite[Proposition 3.1]{DiaGelRam:17}. It is a consequence of ergodicity, partial hyperbolicity, the definition of a Lyapunov exponent, the Brin-Katok theorem, the Birkhoff ergodic theorem, and the Egorov theorem. We refrain from repeating its proof that can be translated \emph{ipsis litteris}. Recall the definition of separated points in~\cite[Chapter 7]{Wal:82}.

\begin{proposition}\label{pro:BriKat}
Let $f\in\cRTPH^1(M)$ and $\mu\in\cM_{\rm erg}(f)$. Let $\alpha=\chi^\c(\mu)$.
Consider continuous functions $\varphi_1,\ldots,\varphi_\ell\colon M \to\bR$.

Then for every $\kappa\in(0,1)$, $r\in(0,1)$,  $\varepsilon_H\in(0,1)$, $\varepsilon_E>0$, and $\varepsilon_B>0$ there exists $\varepsilon_0>0$ such that for every $\varepsilon\in(0,\varepsilon_0)$ there are a positive integer $n_0$ and a subset $\Lambda'\subset M$ satisfying $\mu(\Lambda')>1-\kappa$ such that
 \begin{itemize}
 \item[(1)] there exists $K_0>1$ such that for every $n\ge0$ and every $x\in\Lambda'$ we have
\[
	K_0^{-1}e^{n(\alpha-\varepsilon_E)}
	\le \lVert Df^n|_{E^\c_x}\rVert
	\le K_0 e^{n(\alpha+\varepsilon_E)},
\]
and for every $j=1,\ldots,\ell$, denoting $\overline\varphi_j=\int\varphi_j\,d\mu$, we have
\[
	-K_0+n(\overline\varphi_j-\varepsilon_B)
	\le \sum_{\ell=0}^{n-1}\varphi_j(f^\ell(x))
	\le K_0+n(\overline\varphi_j+\varepsilon_B),
\]
\item[(2)] for every $m\ge n_0$ there is a set  of $(m,\varepsilon)$-separated points $\{x_i\}\subset \Lambda'$ of cardinality $M_m$ satisfying
\[
	M_m
	\ge  L_0^{-1}e^{m(h(\mu)-\varepsilon_H)}.
\]
\end{itemize}
\end{proposition}

\section{Fake invariant foliations and distortion estimates}\label{s.fake}

\subsection{Fake invariant foliations}\label{s.fakefol}

Recall that we are considering a partially hyperbolic diffeomorphisms with a
splitting into three bundles $E^\ss\oplus E^\c \oplus E^\uu$. Let $E^\cs\eqdef E^\ss\oplus E^\c$ and $E^\cu\eqdef E^\c \oplus E^\uu$. 
Recall that the foliations $\cF^\ss$ tangent to $E^\ss$ and $\cF^\uu$ tangent to $E^\uu$ are well defined. However, as we are not assuming dynamical coherence (that is, that the bundle $E^\cs$ and the bundle $E^\cu$ integrate to foliations) we need to find substitutes which serve as foliations (almost) tangent to $E^\ss\oplus E^\c$, $E^\c$, and $E^\c\oplus E^\uu$. For that we use so-called \emph{fake invariant foliations} introduced in \cite{BurWil:10} stated in our context.

 Analogously to our notations above, given a foliation $\widehat\cW$ of some set $B$, for every $x\in B$ and $\rho>0$ denote by $\widehat\cW(x)$ the leaf of this foliation which contains $x$ and by $\widehat\cW(x,\rho)$ the  ball centered at $x$ and with radius $\rho$ in the leaf $\widehat\cW(x)$.

Similarly, as we did in Section~\ref{ss.defblender}, we define cone fields $\cC^\star_\vartheta$ of size $\vartheta$ about $E^\star$ for $\star\in\{\ss,\cs,\c,\cu,\uu\}$. Note that in the case of a unstable blender-horseshoe we had $\cC^\u=\cC^\cu$. 

\begin{proposition}[{\cite[Proposition 3.1]{BurWil:10}}]\label{p.fakefoliations}
	Let $f\in\cRTPH^1(M)$. Then for every $\vartheta>0$ there are constants $\rho>\rho_1>0$ such that for every $p\in M$ the neighborhood $B(p,\rho)$ is foliated by foliations $\widehat \cW^\ss_p$, $\widehat \cW^\c_p$, $\widehat \cW^\uu_p$, $\widehat \cW^\cs_p$, and $\widehat \cW^\cu_p$ with the following properties, for each $\beta \in \{\ss,\cs,\c,\cu,\uu\}$:
\begin{itemize}
\item
{\em{Almost tangency:}}
For each $q\in B(p,\rho)$ the leaf $\widehat \cW^\beta_p(q)$ is $C^1$ and the tangent space of $T_q \widehat \cW^\beta_p$ is in a cone field of size $\vartheta$ about $E^\beta(q)$;
\item
{\em{Local invariance:}} For each $q\in B(p,\rho)$, we have
\[\begin{split} 
	f (\widehat \cW^\beta_{p}(q,\rho_1))
	&\subset \widehat \cW^\beta_{f(p)}(f(q),\rho),\\
	f^{-1} (\widehat \cW^\beta_p(q,\rho_1))
	&\subset \widehat \cW^\beta_{f^{-1}(p)}(f^{-1}(q),\rho).
\end{split}\]
\item
{\em{Coherence:}}
 $\widehat \cW^\ss_p$ and  $\widehat \cW^\c_p$ subfoliate $\widehat \cW^\cs_p$.
 $\widehat \cW^\uu_p$ and  $\widehat \cW^\c_p$ subfoliate $\widehat \cW^\cu_p$. 
\end{itemize}
\end{proposition}

\begin{remark}\label{r.lambda}
Without loss of generality, after possibly changing the metric of $M$ (see~\cite{Gou:07}), if choosing $\vartheta$ sufficiently small, we assume  that there is $\lambda_{\rm fk}>1$ such that for every $p\in M$ and every $v\in\cC^\ss_\vartheta(p)$, $v\ne 0$, we have $\lVert Df_p(v)\rVert\le \lambda_{\rm fk}^{-1}\lVert v\rVert$ and for every $w\in\cC^\uu_\vartheta(p)$, $w\ne 0$, we have $\lVert Df^{-1}_p(w)\rVert\le \lambda_{\rm fk}^{-1}\lVert w\rVert$. 
\end{remark}

\subsection{Distortion in the the central direction}

Given a curve $\gamma$, let 
$$
	\mathrm{Dist} f|_\gamma\eqdef \sup_{x,y \in \gamma} 
		\frac{\lVert Df|_{T _x \gamma}\rVert}
			{\lVert Df|_{T _y \gamma } \rVert}
$$
be the maximal distortion of $f$ in the curve $\gamma$.

Consider the modulus of continuity
$$
	\mathrm{Mod}_\vartheta(\delta) 
	\eqdef  \sup \{  \mathrm{Mod}_\vartheta (\delta, x)\colon      x\in M \},
$$
where
$$
\mathrm{Mod}_\vartheta (\delta, x)\eqdef
\sup \left\{   
\log \frac{\lVert Df|_{{T _x \gamma}}\rVert}{\lVert Df|_{{T _y \gamma}}\rVert}
\colon y \in \gamma, \gamma\in \Gamma^\c_\vartheta (x,\delta)
   \right\},
$$
where $\Gamma_\vartheta^\c(x,\delta)$ denotes the family of curves centered at $x$ of length $2\delta$ and tangent to $\cC^\c_\vartheta$. 
Note that $\Mod_\vartheta(\delta)\to 0$ as $\delta \to 0$.

We need the following distortion control similar to \cite[Corollary 3.5]{DiaGelRam:17}. It  would be immediate if we would have a \emph{true} foliation tangent to the central direction $E^\c$. Since, however, we have to work with fake invariant central foliations and hence need to take into account the variation of their tangent spaces from the true tangent bundle $E^\c$, we provide its proof.

\begin{proposition}[Distortion for zero exponents]
\label{p.distortiondgr}
Given $\vartheta>0$ and $\varepsilon_D>0$ choose $\delta_0$ such that $\mathrm{Mod}_\vartheta (2\delta_0) \le \varepsilon_D$. Given $\varepsilon>0$ and $K>0$, let
$$
	r
	\eqdef \delta_0 K^{-1} e^{-m (\varepsilon+\varepsilon_D)}.
$$
Then for every every $x\in M$ and every $m\ge 1$ such that 
$$
	\lVert  Df^\ell|_{E}  \rVert \le K e^{\ell \varepsilon}
\quad \mbox{for all $\ell\in \{0,\dots,m\}$},
$$
where $E=T_x \widehat \cW_x^\c$ and $\widehat \cW_x^\c$ is the fake invariant foliation associated to the cone field of  size $\vartheta$ about $E^\c$ and $x$, we have
$$
	\lvert \log \mathrm{Dist} f^\ell |_{\widehat\cW^\c_x(x,r)} \rvert 
	\le \ell \varepsilon_D,
\quad \mbox{for every $\ell\in \{0,\dots,m\}$.}
$$
\end{proposition}

\begin{proof}
Let $Z=\widehat\cW^\c_x(x,r)$. Let us denote by $\lvert\gamma\rvert$ the length of a curve $\gamma$.
	The proof is by (finite) induction on $\ell$. Note that the claim holds for $\ell=0$. Suppose that the claim holds for every $\ell=0,\ldots,i$ for some $i>0$. This means that we have $\lvert\log \dist f^\ell|_Z\rvert\le \ell\varepsilon_D$ for every $\ell\in\{0,\ldots,i\}$, which by the hypothesis of the proposition implies that
\[\begin{split}
	\lvert f^i(Z)\rvert
	&\le 	\lVert Df^i|_{T_xZ}\rVert 
			\cdot \dist f^i|_{Z} \cdot \lvert Z\rvert\\
	&\le Ke^{i\varepsilon}		
			\cdot e^{i\varepsilon_D} 
			\cdot r
	\le Ke^{i\varepsilon}		
			\cdot e^{i\varepsilon_D} 
			\cdot \delta_0 K^{-1} e^{-m (\varepsilon+\varepsilon_D)}		\\		
	&= \delta_0e^{-(m-i)(\varepsilon+\varepsilon_D)}			
	\le \delta_0.
\end{split}\]	
Hence $\lvert\log\dist f|_{f^i(Z)}\rvert\le \varepsilon_D$. Now we apply the chain rule and obtain $\lvert\log\dist f^{i+1}|_{f^{i+1}(Z)}\rvert\le i\varepsilon_D+\varepsilon_D$, which is the claim for $i+1$. We can repeat these arguments until $i=m$.
\end{proof}

\subsection{Distortion in the stable and unstable direction}

The next lemma is a standard consequence of uniform expansion/contraction along un-/stable (fake) foliations and sometimes referred to as tempered distortion. 
See for instance the proof of \cite[Lemma 2.4]{BocBonDia:16} in a similar context%
\footnote{There, the proof is stated for so-called flip-flop families and the only property required is that they are tangent to un-/stable expanding/contracting cone fields.}.

\begin{lemma}\label{l.strongdistortion}
Given $\vartheta>0$, let $\rho>\rho_1>0$ be as in Proposition~\ref{p.fakefoliations}.   For every $\varepsilon_D^\ss,\varepsilon_D^\uu>0$, there is $m_0\ge 1$ such that for every $m\ge m_0$ we have:
\begin{itemize}
\item for $x,y,p\in M$ satisfying $f^\ell(y)\in\widehat\cW^\uu_{f^\ell(p)}(f^\ell(x),\rho_1)$ for all $\ell\in\{0,\ldots,m\}$ we have 
\[
	\Big\lvert \log\frac{\lVert D f^\ell|_{T_x\widehat\cW^\c_p}\rVert}
					 	{\lVert D f^\ell|_{T_y\widehat\cW^\c_p}\rVert}\Big\rvert
	\le \ell\varepsilon_D^\uu.
\]
\item for $x,y,p\in M$ satisfying $f^{-\ell}(y)\in\widehat\cW^\ss_{f^{-\ell}(p)}(f^{-\ell}(x),\rho_1)$ for all $\ell\in\{0,\ldots,m\}$ we have 
\[
	\Big\lvert \log\frac{\lVert D f^{-\ell}|_{T_x\widehat\cW^\c_p}\rVert}
					 	{\lVert D f^{-\ell}|_{T_y\widehat\cW^\c_p}\rVert}\Big\rvert
	\le \ell\varepsilon_D^\ss.
\]
\end{itemize}
\end{lemma}

\section{Construction of the hyperbolic sets $\Gamma^{\pm}$ in Theorem~\ref{t.approx}}\label{sec:5THEPROOF}

In this section we will prove the following result.

\begin{theorem}\label{teo:finally}
	Assume that $f$ is a $C^1$-diffeomorphism with a partially hyperbolic splitting $TM=E^\ss \oplus E^\c\oplus E^\uu$ with three non-trivial bundles such that $E^\ss$ is uniformly contracting, $E^\c$ is one-dimensional, and $E^\uu$ is uniformly expanding such that the strong stable and the strong unstable foliations are both minimal and $f$ has an unstable blender-horseshoe for $f^{n}$ for some $n\ge1$. 
	Then every nonhyperbolic ergodic measure $\mu$ of $f$ has the following properties. For every $\delta>0$ and every $\gamma>0$ there exists a basic set $\Gamma^+$ being central expanding whose topological entropy satisfies
\[	
	h_{\rm top}(f,\Gamma^+) \in  [h(\mu)-\gamma, h(\mu)+\gamma].
\]	
Moreover, every measure $\nu^+ \in\cM(f,\Gamma^+)$  is $\delta$-close to $\mu$. In particular, there are hyperbolic measures $\nu^+\in\cM_{\rm erg}(f,\Gamma^+)$ satisfying
$$
	\chi(\nu^+)\in(0,\delta)
	\quad\text{ and }\quad
	h(\nu^+)\in [h(\mu)-\gamma, h(\mu)+\gamma].
$$
\end{theorem}

There is the corresponding result claiming the existence of a central contracting basic set under the assumption that there is a stable blender-horseshoe. Its proof follows by considering $f^{-1}$ instead of $f$.

Note that Theorem~\ref{t.approx} is an immediate consequence of the above theorem. For that just recall that every $f\in\cORTPH^1(M)$, by definition, has a stable and an unstable blender-horseshoe and the strong foliations are both minimal (see Proposition~\ref{p.blenderopendense}).

In the course of this section, given a diffeomorphism $f$ satisfying the hypotheses of Theorem~\ref{teo:finally} and a nonhyperbolic ergodic measure $\mu$, we will construct the basic set $\Gamma^{+}$ claimed in this theorem. This section is organized as follows. In the preliminary Section~\ref{ss:prelim} we collect and fix some quantifiers from previous sections.
In Section~\ref{skeleleton} we introduce so-called skeletons.
In Section~\ref{sss.sizefakecentral}, we complete the preparatory choice of quantifiers.
 In Section~\ref{ss:Gamma}, we define the set $\Gamma^+$. Its construction is geometrical and involves the results in previous sections. In Section~\ref{s.hypGamma}, we see that $\Gamma^+$ is a hyperbolic set with stable index  $s$ and entropy close to the one of $\mu$ (see Proposition~\ref{p.Gammahyp}). In Section~\ref{ss.birkhoff}, we see that the ergodic measures supported on $\Gamma^+$ are close to $\mu$ (see Proposition~\ref{p:weak}).

\subsection{Preliminaries}\label{ss:prelim}

Consider $f$ as in Theorem~\ref{teo:finally}, let $\mu\in\cM_{\rm erg,0}(f)$. Denote $h=h(\mu)$. 

Fix an unstable blender-horseshoe $(\Lambda_f^+,\mathbf{C}_f^+,f^{n_+})$ given by Proposition~\ref{p.blenderopendense}. 
Fix open sets $V^+,N^+,U^+\subset\mathbf C_f^+$ in-between the $\uu$-walls of the blender-horseshoe satisfying $V^+\subset\overline{V^+}\subset N^+\subset\overline{N^+}\subset U^+$.
Consider also a safety domain $\widehat{\mathbf C}^+_f$ and an associated safety constant $\safe=\safe({\mathbf C}^+_f,\widehat{\mathbf C}^+_f,V^+,N^+)>0$ as in Remark~\ref{reml.minimalitysizesafety}. 
The blender is endowed with cone fields $\cC^\star_{\vartheta_0}$, $\star\in\{\ss,\cs,\c,\cu,\uu\}$, of opening $\vartheta_0>0$ arbitrarily small, that extend to the whole manifold $M$.
Fix also $\theta>0$ as in Remark~\ref{rmlem:openneigh} (note that this constant implicitly depends on the opening of the cone fields fixed above).

For simplicity, in what follows, we assume that $n^+=1$. 

Let us denote 
\begin{equation}\label{eq:defmdf}
\begin{split}	
	\mathfrak m
		&\eqdef \min\{\|Df_x(v)\|\colon x\in M, v\in T_xM,\|v\|=1\},\\
	\mathfrak M
		&\eqdef \max\{\|Df_x(v)\|\colon x\in M, v\in T_xM,\|v\|=1\}.		
\end{split}\end{equation}	

Fix small numbers $\varepsilon_H>0$, $\varepsilon_E>0$, and $\eps_B>0$. Note that there is a finite set $\Phi=\{\varphi_j\}$ of continuous potentials over $M$ such that if a probability measure $\nu$ satisfies 
\begin{equation}
\label{eq:choiceofeB}
	\max_{\varphi_j\in\Phi}\left\lvert\int\varphi_j\,d\nu-\int\varphi_j\,d\mu\right\rvert
	<\eps_B
\end{equation}
then the distance between $\nu$ and $\mu$ is smaller than $2\eps_B$.

Choose $\vartheta\in(0,\vartheta_0)$  such that the variation of $Df$ in a $\vartheta$-cone field about $E^\c$ is bounded by $\varepsilon_E$, that is for every $x\in M$ 
\begin{equation}\label{eq:choisecone}
	\sup_{v,w\in \cC^\c_\vartheta(x),\lVert v\rVert=\lVert w\rVert=1}
	\Big|\log\frac{\lVert Df_x(v)\rVert}{\lVert Df_x(w)\rVert}\Big|
	<\varepsilon_E,
\end{equation}

For convenience, let us first restate Corollary~\ref{cor:nofim} in the setting of this section. 

\begin{remark}[Quantifiers in fake invariant and true foliations]\label{newr.l.inbetween}
Given  $\vartheta>0$,  consider the fake invariant foliation $\widehat \cW^\uu_p$, $p\in M$, associated to $\cC_\vartheta^\uu$ and the associate numbers $\rho>\rho_1>0$ given in Proposition~\ref{p.fakefoliations}.
Recall also the definition of  the expansion (and contraction) constant $\lambda_{\rm fk}>1$ along strong fake curves as in Remark~\ref{r.lambda}.
\end{remark}

We now fix the constants related to distortion properties.
Given  $\vartheta$ as above, fix $\varepsilon_D>0$ sufficiently small and let $\delta_0$ satisfy as in Proposition~\ref{p.distortiondgr} 
\begin{equation}\label{eq:choosedelta0}
	\Mod_\vartheta(2 \delta_0)
	\le \varepsilon_D
	\quad\text{ and also }\quad
	\delta_0\in(0,\rho_1).
\end{equation}
Moreover, given $\rho>\rho_1>0$ as in Remark~\ref{newr.l.inbetween},  we choose  $\eps_D^\ss>0$ and $\eps_D^\uu>0$ sufficiently small and we let $m_0$ as in Lemma~\ref{l.strongdistortion}.
We now let 
 \begin{equation}\label{eq:epsilonhattilde}
 \eps_2 \eqdef  2\varepsilon_E+\eps_D^\ss+\eps_D^\uu+\eps_D
 \quad \mbox{and}\quad
 \eps_1\eqdef \eps_2+ \eps_D^\ss
\end{equation}
and assume that $\varepsilon_E, \varepsilon^\ss_D, \varepsilon_D^\uu,\varepsilon_D$ were chosen small enough such that
\[
%	\max\{\eps_1,\eps_2\}=\eps_1	
	\eps_1
	\ll \log\lambda_{\rm fk}.
\]

\subsection{Skeletons}\label{skeleleton}
We now choose the ``skeleton'' for the construction of our basic set. Here, 
we reformulate the \emph{skeleton property} in \cite[Section 4]{DiaGelRam:17} to our partially hyperbolic context. It is a consequence of the interplay of  Proposition~\ref{pro:BriKat} which is entirely based on ergodic theory and the minimality of the strong un-/stable foliations which is a purely topological property.

For the first part, concerning the ergodic properties of our nonhyperbolic ergodic measure $\mu$, by Proposition~\ref{pro:BriKat}, for $\varepsilon_H$, $\varepsilon_E$, and $\varepsilon_B>0$ chosen as above and with $h=h(\mu)$ and $\alpha=\chi(\mu)$ there exists $\varepsilon_0>0$ such that for every $\varepsilon\in(0,\varepsilon_0)$ there are constants $K_0,L_0\geq 1$ and an integer $n_0\geq 1$ such that for every $m\geq n_0$ there exists a finite set  $\mathfrak{X}=\{x_i\}$ of $(m,\varepsilon)$-separated points satisfying the following:
\begin{itemize}
	\item $\card\mathfrak{X}\geq L_0^{-1}e^{m (h-\varepsilon_H)}$;
	\item for every $\ell=0,\ldots,m$ and every $i$ one has 
	\begin{equation}\label{e.derivadacentral}
	K_0^{-1}e^{-\ell(\alpha+\varepsilon_E)}\leq
	\lVert Df^\ell |_{E^\c_{x_i}} \rVert 
	\leq K_0e^{\ell(\alpha+\varepsilon_E)},
	\end{equation}
	recall that we assume $\alpha=\chi(\mu)=0$;%
	\footnote{We will use the case $\alpha<0$ in Section~\ref{sec:theo:main3twin}.}
	\item for every $\varphi_j\in\Phi$, denoting $\overline\varphi_j=\int\varphi_j\,d\mu$, for every $n\ge0$ and every $i$ we have
	\begin{equation}\label{e.birrrrkhoff}
	-K_0+n(\overline\varphi_j-\varepsilon_B)
	\le \sum_{\ell=0}^{n-1}\varphi_j(f^\ell(x_i))
	\le K_0+n(\overline\varphi_j+\varepsilon_B).
	\end{equation}
\end{itemize}  

For the following, fix now the quantifier $\varepsilon\in(0,\varepsilon_0)$ and the hence associated constants $K_0,L_0\ge1$ and $n_0$ as described above. 

 Fix now $\delta >0$ satisfying
\begin{equation}\label{e.niteroiuff}
	\delta
	<\min\Big\{\frac{\varepsilon}{15},\frac{\rho_1}{10}\Big\}.
\end{equation}

To complete the second part of the skeleton,  concerning minimality, with the quantifiers chosen in Remark~\ref{newr.l.inbetween}, consider a \emph{connecting time} $t_{\rm con}\eqdef t_{\rm con}(V^+,\vartheta,\delta/4)\ge1$ as in Corollary~\ref{cor:nofim} and Remark~\ref{reml.minimalitysizesafety}. Then for every $p\in M$, every $x\in B(p,\rho_1)$, and every $t\ge t_{\rm con}$
\begin{itemize}
\item $f^t ( \widehat \cW^\uu_p(x,\delta/4)) $ contains a $\uu$-disk contained in $V^{+\uu}_{\widehat{\mathbf C}^+_f}$,
\item $f^{-t} ( \cF^\ss(x,\delta/4))$ contains a $\ss$-disk contained in $V^{+\ss}_{\widehat{\mathbf C}^+_f}$,
\end{itemize}
where in both cases the disks are complete relative to $\widehat{\mathbf C}^+_f$.

Finally, to conclude, we make an appropriate choice of $m\ge n_0$. Let us first consider the constants
$$
	\mathfrak{N}
	\eqdef\mathfrak M^{-t_{\rm con}}\min\{ \theta,\frac\tau4,
				\frac\varepsilon{15}, \delta_0\}.
%				,
%\quad
%	\textcolor{blue}{\mathfrak{K}
%	\eqdef \frac{\log K_0+ \lvert\log \delta\rvert}
%			{\log \lambda_{\mathrm{fk}}+4 \eps_2}.}
$$
In what follows we fix now the remaining quantifier $m$ satisfying
\begin{equation} 
\label{eq:firstchoiceof}
	m\ge \max 
	\Big\{m_0,n_0,
	 \frac{\lvert\log \mathfrak{N}\rvert}{3\eps_1}, 	
		\frac{\log 12+\log\rho_1+\log K_0}
				{\log\lambda_{\rm fk}-(\sqrt{\eps_1}+\eps_1)}
		%\mathfrak{K}	 
		\Big\}.
\end{equation}

After these choices, we fix 
\begin{equation}\label{eq:skeleleton}
	\mathfrak{X}
	=\mathfrak{X}(h,\alpha,\delta,\varepsilon_H,\varepsilon_E,\varepsilon_B,\varepsilon,m)
	\eqdef \{x_i\}
\end{equation}
as above.	 

\subsection{Postliminaries}\label{sss.sizefakecentral}
To conclude the choices of quantifiers, let us now choose a size of fake central curves.
Given $m$ as in \eqref{eq:firstchoiceof} and $\eps_1$ as in~\eqref{eq:epsilonhattilde}, we define 
\begin{equation}\label{eq:formulapetrea}
	\delta_\c
	\eqdef e^{-m\sqrt{\eps_1}}. 
\end{equation}
For further reference, observe that with \eqref{eq:firstchoiceof} we have the following 
\begin{equation}\label{eq:effect}
	\delta_\c
	< \mathfrak M^{-t_{\rm con}}\min\{ \theta,\frac\tau4,
			\frac{\varepsilon}{15},\delta_0,\delta\}
\end{equation}
and with~\eqref{e.niteroiuff} and~\eqref{eq:firstchoiceof}
\begin{equation}\label{eq:fakem}
	\lambda_{\rm fk}^{-m}\delta
	< \lambda_{\rm fk}^{-m}\rho_1
	< \frac{1}{12}\delta_\c K_0^{-1} e^{-m\eps_1}
	< \frac{1}{12}\delta_\c K_0^{-1} e^{-m\eps_2}.
\end{equation}

\subsection{Construction of the set $\Gamma^+$}\label{ss:Gamma}

The set $\Gamma^+$ is obtained as follows. We will construct disjoint full rectangles $C_i$, $i\in \{1,\dots, \card \mathfrak{X}\}$, of the form
$$
C_i=\bigcup_{x\in R_i^\cu} \cF^\ss_{\mathbf C^+_f}(x),
$$
where $R_i^\cu$ is a small disk in some fake center unstable set contained in the domain $\mathbf C^+_f$ of the blender-horseshoe. We see that there is some $N\sim m$, independent of $i$,  such that $f^N(C_i\cap R_i^\cu)$ is a $\c$-complete $\u$-strip of the blender-horseshoe.%
\footnote{In the following we will write $\psi_1(m)\sim \psi_2(m)$ if $\lim_{m\to\infty}\psi_1(m)/\psi_2(m)=1$.} %
  This will imply that $f^N(C_i)$ intersects in a ``Markovian way" each $C_j$. We will define $\Gamma^+$ as the orbit of the maximal invariant set of $f^N$ in the union of these rectangles, see Section~\ref{sss.Gamma}. We see that, for $x\in \Gamma^+$, $Df^N|_{E^\cu_x}$ is uniformly expanding, $Df^N|_{E^\ss_x}$ is uniformly contracting, and the central exponents of points in $\Gamma^+$ are small. Moreover, the restriction of $f^N$ to $\Gamma^+$ is conjugate to the full shift on $\card\mathfrak{X}$ symbols. Thus, since $N\sim m$, this set has entropy close to $h(\mu)$. For details see Proposition~\ref{p.Gammahyp} in Section~\ref{s.hypGamma}.
Finally, in Section~\ref{ss.birkhoff} we see that the ergodic measures supported on $\Gamma^+$ are close to $\mu$.

We now explain the construction of the rectangles $R^\cu_i$. Each rectangle $R^\cu_i$ is obtained
foliating a central fake curve $\gamma_i$ (centred at some point $y_i$) by small fake $\uu$-sets.
The orbit of $y_i$ ``shadows" during $m$ iterates the point $x_i$ in the skeleton. From this we get that, during these $m$ iterates, ``their central derivatives" are close. Thereafter, and after a controlled time, this point lands in the domain of the blender. Similarly, the point $y_i$ after a controlled time backward   lands in that domain.
The choice of the point $y_i$ is done in Section~\ref{sss.beginning}.
To construct the curve $\gamma_i$ we need some estimates of the central derivative, see Section~\ref{sss.centralestimates}. After some preliminary  constructions (auxiliary rectangles in Section~\ref{sec:fakecurect} and blending-like properties in Section~\ref{sss.blending}), the construction of the rectangles $R^\cu_i$ is completed in 
Section~\ref{sss.finalconstruction}. Finally, in Section~\ref{sss.separation}, we see that the orbits of points in different full rectangles $C_i$ are sufficiently separated.

\subsubsection{Beginning of the construction}\label{sss.beginning}
Now we apply the connecting properties in Section~\ref{skeleleton} to each point of the skeleton $\mathfrak{X}=\{x_i\}$ in~\eqref{eq:skeleleton} as follows. 
We will consider fake invariant foliations $\widehat\cW^\star_{f^\ell(z_i)}$,  $\star\in\{\c,\cu,\uu\}$, for $\ell=0,\ldots,m$, where $\widehat\cW^\cu$ is subfoliated by $\widehat\cW^\c$ and $\widehat\cW^\uu$, respectively. We will also consider the (true) strong stable foliation $\cF^\ss$.

Recall that there is $\Delta^\ss\subset \cF^\ss(x_i,\delta/4)$ such that $f^{-t_{\rm con}}(\Delta^\ss)$ is a $\ss$-disk contained in $V^{+\ss}_{\widehat{\mathbf C}^+_f}$ complete relative to $\widehat{\mathbf C}^+_f$. This implies that there exists $z_i\in\Delta^\ss$ such that 
\[
	f^{-t_{\rm con}}(z_i)\in V^+_{\widehat{\mathbf C}^+_f}
	\quad\text{ and }\quad 
	f^{-t_{\rm con}}(\widehat\cW^\uu_{z_i}(z_i,\delta/4))
		\subset N^{+\uu}_{\widehat{\mathbf C}^+_f}.
\]	 
Observe now that there is $\Delta^\uu\subset\widehat\cW^\uu_{z_i}(z_i,\delta/4)$ such that 
\[
	f^\ell(\Delta^\uu)\subset\widehat\cW^\uu_{f^\ell(z_i)}(f^\ell(z_i),\delta/4)
	\quad\text{ for every }\ell\in\{0,\ldots,m\}
\]
 and 
 \[
 	f^m(\Delta^\uu)= \widehat\cW^\uu_{f^m(z_i)}(f^m(z_i),\delta/4).
\]	
Hence, again by Remark~\ref{newr.l.inbetween} and the connecting properties  for the skeleton in Section~\ref{skeleleton}, there is $\Delta^\uu_0\subset f^m(\Delta^\uu)$ such that $f^{t_{\rm con}}(\Delta^\uu_0)$ is a $\uu$-disk contained in $V^{+\uu}_{\widehat{\mathbf C}^+_f}$ complete relative to $\widehat{\mathbf C}^+_f$.
Finally, choose any point $y_i\in f^{-m}(\Delta^\uu_0)$. These considerations prove the following lemma (compare also Figure~\ref{Fig:plug}).

\begin{lemma}\label{l.choosingpoints}
For every $i=1,\ldots,\card\mathfrak{X}$ there exist points $z_i$ and $y_i$ having the following properties:
\begin{enumerate}
	\item $f^{-t_{\rm con}}(y_i)\in N^+\subset\mathbf C^+_f$.
	\item $f^{t_{\rm con}+m}(y_i)\in\mathbf{C}_f^+$ is  in  some $\uu$-disk complete relative to $\widehat{\mathbf C}^+_f$ in $V^{+\uu}_{\widehat{\mathbf C}^+_f}$.
	\item $z_i\in \cF^{\ss}(x_i,\delta/4)$.
	\item $f^\ell(y_i)\in\widehat\cW^\uu_{f^\ell(z_i)}(f^\ell(z_i),\delta/4)$, for each $\ell=0,\ldots,m$. 
\end{enumerate}
\end{lemma}

\begin{figure}[h] 
\begin{overpic}[scale=.5, %grid, tics=10
]{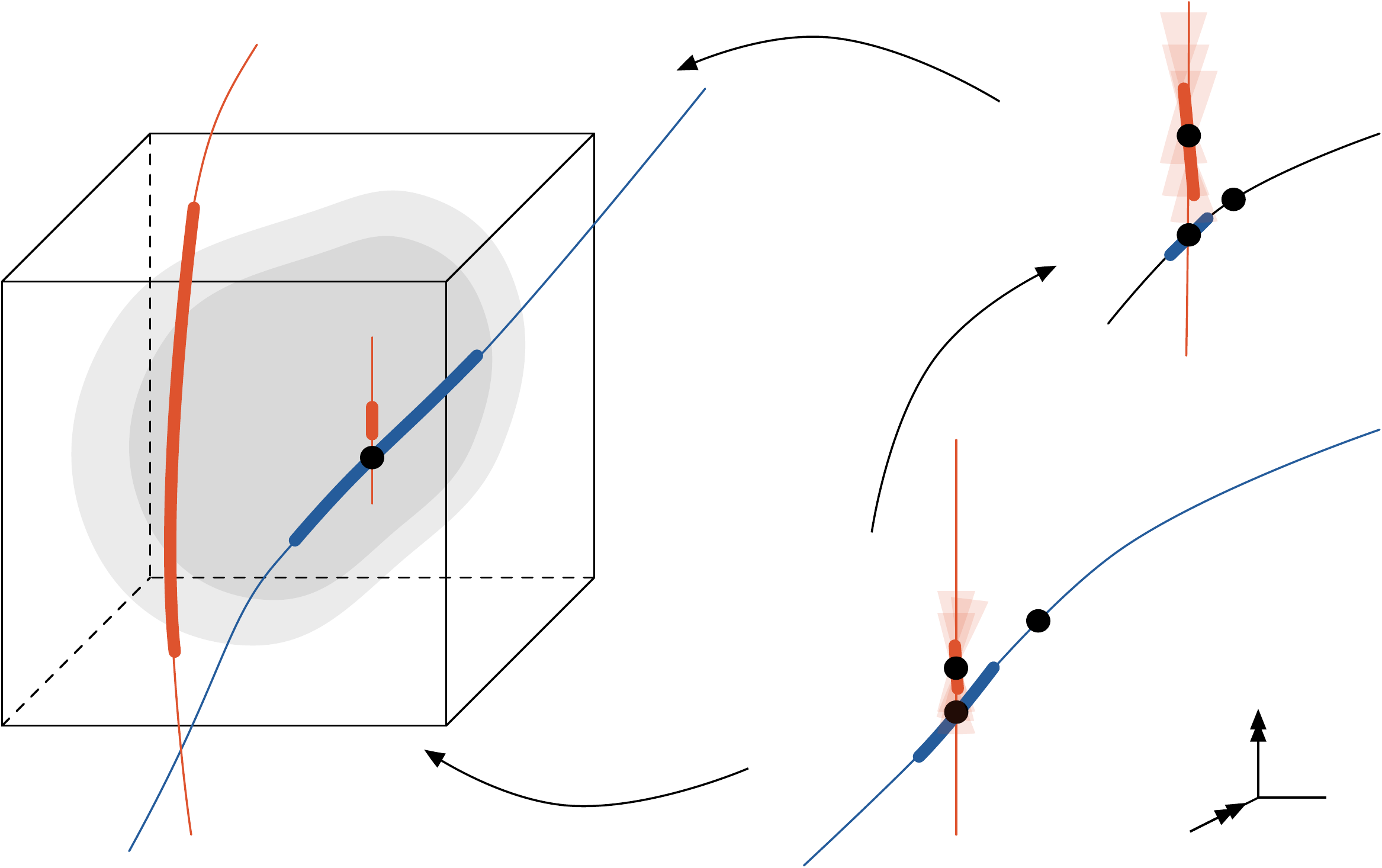}
\put(23,47){{\small{$N^+$}}}
\put(20,43){{\small{$V^+$}}}
\put(0,5){{\small{$\mathbf C^+_f$}}}
\put(64,40){{\small{$f^m$}}}	
\put(40,0){{\small{$f^{-t_{\rm con}}$}}}	
\put(56,55){{\small{$f^{t_{\rm con}}$}}}	
 \put(77,16){{\small{$x_i$}}}
 \put(85,30){{\small{$\cF^\ss(x_i)$}}}
 \put(71,10){{\small{$z_i$}}}
 \put(70.5,16){{\small{$y_i$}}}
 \put(61,8){{\small{\textcolor{blue}{$\Delta^\ss$}}}}
 \put(40,13){{\small{{$\widehat\cW_{z_i}(z_i)\supset f^{-m}(\Delta_0^\uu)$}}}}
 \put(87,52){{\small{$f^m(y_i)$}}}
 \put(87,43){{\small{$f^m(z_i)$}}}
 \put(91,47){{\small{$f^m(x_i)$}}}
 \put(83,0){{\small{$E^\ss$}}}
 \put(92,12){{\small{$E^\uu$}}}
 \put(97,5){{\small{$E^\c$}}}
 \end{overpic}
 \hspace{2cm}
\caption{Construction of rectangles: choosing reference points}
\label{Fig:plug}
\end{figure}
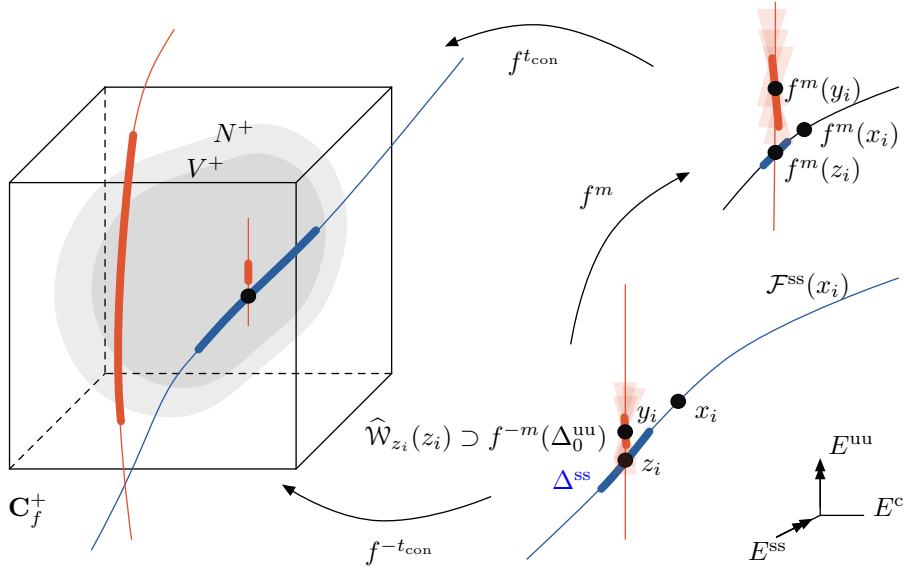

\subsubsection{Distortion estimates nearby the orbit of $y_i$}
\label{sss.centralestimates}

Recalling~\eqref{e.derivadacentral} and applying the second part of the distortion Lemma~\ref{l.strongdistortion} to $x_i$ and $z_i$, for every $\ell=0,\ldots,m$ we get
$$
	 K_0^{-1}e^{\ell (-\varepsilon_E-\eps_D^\ss)}\leq
	\lVert  Df^\ell |_{E^\c_{z_i}} \rVert 
	\leq  K_0 e^{\ell (\varepsilon_E+\eps_D^\ss)}.
$$
Then applying the first part of Lemma~\ref{l.strongdistortion} to $z_i$ and any   $y\in \widehat \cW^\uu_{z_i} (z_i, \delta/4)$, for every $\ell=0,\ldots,m$ we get
$$
	 K_0^{-1}e^{\ell (-\varepsilon_E-\eps_D^\ss-\eps_D^\uu)}\leq
	\lVert  Df^\ell |_{E^\c_{y}} \rVert 
	\leq  K_0 e^{\ell (\varepsilon_E+\eps_D^\ss+\eps_D^\uu)}.
$$
With~\eqref{eq:choisecone},  for every $\ell=0,\ldots,m$ we have 
\begin{equation}\label{eq:afomurla}
	 K_0^{-1}e^{\ell (-2\varepsilon_E-\eps_D^\ss-\eps_D^\uu)}\leq
	\lVert  Df^\ell |_{T_{y}\widehat\cW^\c_{z_i}} \rVert 
	\leq  K_0 e^{\ell (2\varepsilon_E+\eps_D^\ss+\eps_D^\uu)}.
\end{equation}
We now apply Proposition~\ref{p.distortiondgr} to $\vartheta$ and $\varepsilon_D$ with $\delta_0$ as in~\eqref{eq:choosedelta0} satisfying $\Mod_\vartheta(\delta_0)
	\le \varepsilon_D$ and to $\varepsilon=2\varepsilon_E+\eps_D^\ss+\eps_D^\uu$  
and $K=K_0$ and let
\begin{equation}\label{def:r}
	r
	\eqdef \frac12 \delta_\c K_0^{-1}e^{-m \eps_2}
			%(2\varepsilon_E+\eps_D^\ss+\eps_D^\uu+\eps_D)}
	< \frac12\delta_0 K_0^{-1}e^{-m \eps_2},
			%(2\varepsilon_E+\eps_D^\ss+\eps_D^\uu+\eps_D)},
\end{equation}
where we use that $\delta_\c< \delta_0$, see \eqref{eq:effect}.
Hence, by this proposition, together with~\eqref{eq:afomurla}, for every $y\in \widehat \cW^\uu_{z_i} (z_i, \delta/4)$ we obtain that for every $x\in\gamma\eqdef\widehat\cW^\c_{z_i}(y,2r)$ for every $\ell=0,\ldots,m$  we have 
\begin{equation}\label{eq:pascoa}
	K_0^{-1} e^{-\ell \eps_2} % (2\varepsilon_E+\eps_D^\ss+\eps_D^\uu+\eps_D)}
	\le \lVert Df^\ell|_{T_x\gamma}\rVert
	\le K_0 e^{\ell \eps_2}. %(2\varepsilon_E+\eps_D^\ss+\eps_D^\uu+\eps_D)}.
\end{equation}

\begin{remark}\label{r:maisumacentral}
Take any $y \in \widehat{\cW}^\uu_{z_i}(z_i,\delta/4)$ any $x\in \widehat \cW^\c_{z_i} (y,2r)$. Applying again  Lemma~\ref{l.strongdistortion} as above, we get the following estimate for every $z\in \cF^\ss(x,\delta/4)$
$$
K_0^{-1} e^{-\ell\eps_1}
	\le \lVert Df^\ell|_{E^\c_z}\rVert
	\le K_0 e^{\ell \eps_1},
$$
for every $\ell=0,\ldots,m$, where $\eps_1$ is as in \eqref{eq:epsilonhattilde}.
\end{remark}

\subsubsection{Fake central curves and auxiliary center-unstable rectangles}\label{sec:fakecurect}

We now study the deformation of the fake central curves under forward iterations of $f$. Consider  the central curve $\gamma_i=\widehat \cW^\c_{z_i} (y_i, r)$. Consider its image $f^m(\gamma_i)$ and denote its extreme points by $a$ and $b$. Consider the $\uu$-disks 
 \[
 	\widehat\Delta^\uu(p)\eqdef  \widehat\cW^\uu_{f^m(z_i)}(p,\delta/4), 
	\quad p\in\{a,b\}.
\]	 
Since by~\eqref{e.niteroiuff},~\eqref{def:r},~\eqref{eq:formulapetrea}, and~\eqref{eq:firstchoiceof} we have $\frac\delta2+r+\frac\delta2\le\rho_1$ these disks are indeed contained in the fake leaf $\cW^\cu_{f^m(z_i)}(f^m(z_i),\rho_1)$ and hence are well defined. 
Now take their pre-images 
\[
	\widetilde \Delta^\uu(f^{-m}(a))\eqdef f^{-m}(\widehat\Delta^\uu(a))
	\quad\text{ and }\quad
	\widetilde \Delta^\uu(f^{-m}(b))\eqdef f^{-m}(\widehat\Delta^\uu(b)).
\]	 
Recalling the choice of $\lambda_{\rm fk}>1$ in Remark~\ref{r.lambda}, these disks have diameter at most $\lambda_{\rm fk}^{-m}\delta$ and are tangent to $\cC^\uu_\vartheta$.
By our choice of $m$ in \eqref{eq:fakem}, we have $\lambda_{\rm fk}^{-m}\delta< r/6$. In this way, for every point $\widehat \Delta^\uu(f^{-m}(a))$ there is a fake central curve which ends in $\widehat\cW^\uu_{z_i}(f^{-m}(b),\delta/4)$ and for every point $\widehat \Delta^\uu(f^{-m}(b))$ there is a fake central curve which ends in $\widehat\cW^\uu_{z_i}(f^{-m}(a),\delta/4)$.
Hence, we get that for every fake central curve which starts in $\widehat \Delta^\uu(f^{-m}(a))$ and ends in $\widehat \Delta^\uu(f^{-m}(b))$ 
 is well defined and has length between $\frac12 r$ and $2r$ (compare Figure~\ref{Fig:def-L}). 
Denote by $\cL$ the set consisting of all such fake central curves. 
Taking the union of all such curves in $\cL$, we define the center-unstable rectangle $R^\cu_{i,0}$ by
\[%\begin{equation}\label{eq:provisoryrect0}
	R^\cu_{i,0}
	\eqdef \bigcup_{\gamma\in\cL}\gamma	.
\]%\end{equation}

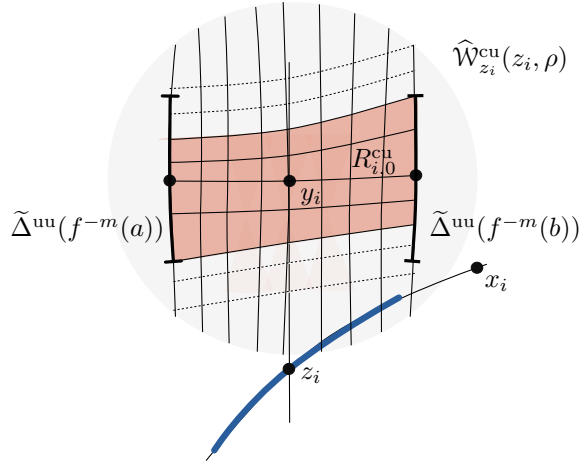
\begin{figure}[h] 
\begin{overpic}[scale=.4, %grid, tics=10
]{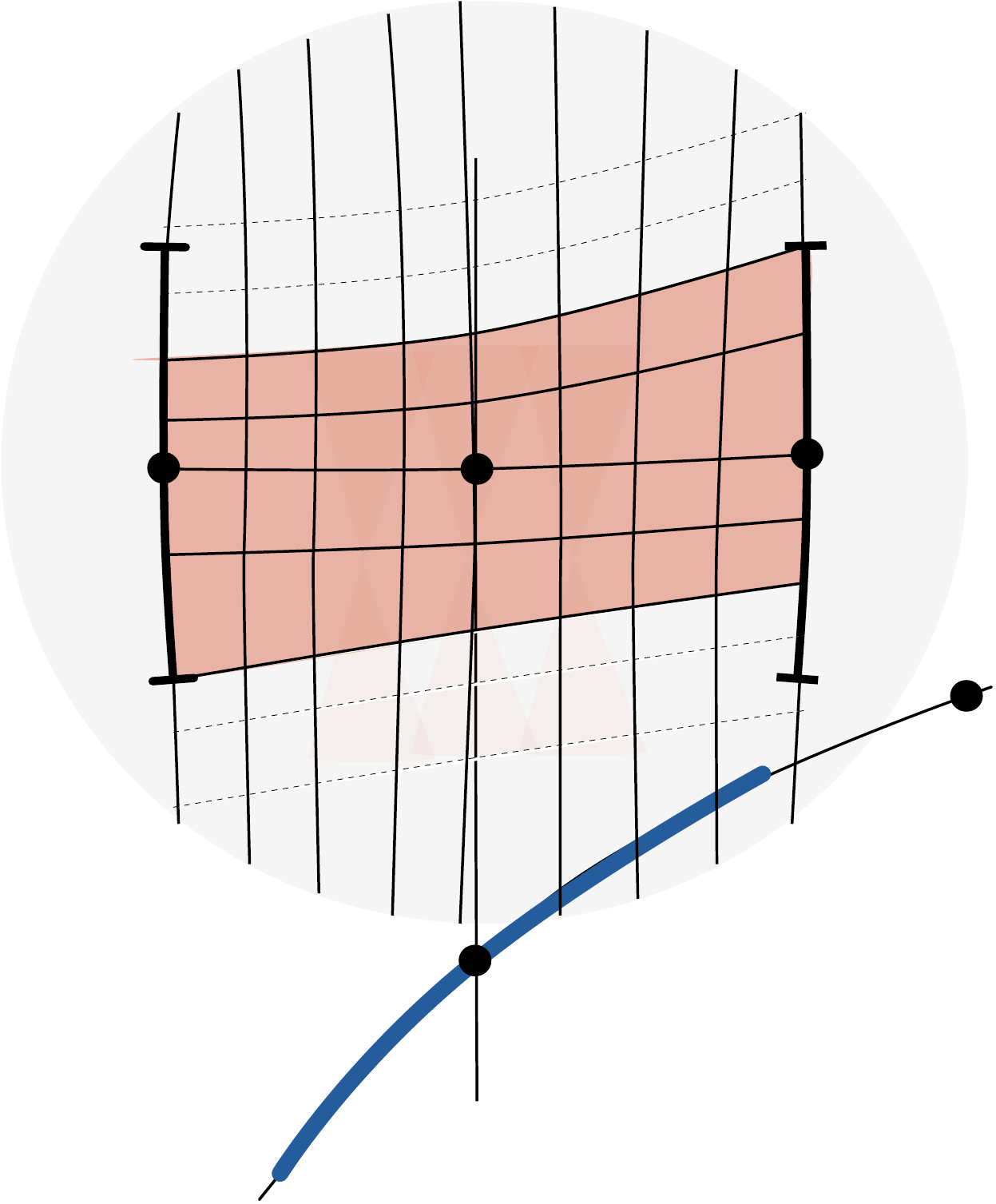}
\put(-21,49){{\small{$\widetilde\Delta^\uu(f^{-m}(a))$}}}
\put(70,49){{\small{$\widetilde\Delta^\uu(f^{-m}(b))$}}}
\put(42,18){{\small{$z_i$}}}
\put(82,38){{\small{$x_i$}}}
\put(42,57){{\small{$y_i$}}}
\put(75,86){{\small{$\widehat\cW^\cu_{z_i}(z_i,\rho)$}}}
\put(53,64){{\small{$R^\cu_{i,0}$}}}
 \end{overpic}
 \hspace{2cm}
\caption{Definition of the center-unstable rectangle $R_{i,0}^\cu\subset \widehat\cW_{z_i}(z_i,\rho)$}
\label{Fig:def-L}
\end{figure}

Finally, for each $\ell \in \{0,\dots, m\}$, we let
%\begin{equation}
\[%\label{eq:provisorycurectell}
	R^\cu_{i,\ell} 
	\eqdef f^\ell (R^\cu_{i,0}).
\]%\end{equation}
Consider now the family of curves $f^\ell(\zeta)$, $\zeta\in\cL$ and $\ell=0,\ldots,m$. We claim that all of them are fake central curves contained in $\widehat\cW^\cu_{f^\ell(z_i)}(f^\ell(z_i),\rho_1)$.
Indeed, combining Lemma~\ref{l.strongdistortion} with the estimates \eqref{eq:pascoa}, with the notation \eqref{eq:epsilonhattilde} every curve $f^\ell(\zeta)$, $\zeta\in\cL$, has length between 
\[%\begin{equation}\label{eq:quintasanta2}
	\frac12r\cdot K_0^{-1} e^{-\ell \eps_2}  
	\le \lvert f^\ell(\zeta)\rvert
	\le 2 r\cdot K_0 e^{\ell \eps_2} .
\]%\end{equation}
Hence, by the choice of $r$ in~\eqref{def:r},  we have
\begin{equation}
\label{eq:estimatedeltac}
	\frac14\delta_\c K_0^{-2}
	e^{-m\eps_2}	
	e^{-\ell\eps_2}
	\le \lvert f^\ell(\zeta)\rvert
	\le  \delta_\c
	e^{-m\eps_2}e^{\ell\eps_2}
	\le \delta_\c.
\end{equation}
Now recall that $\delta_\c<\delta_0<\rho_1$, see~\eqref{eq:choosedelta0} and~\eqref{eq:effect}. Arguing exactly as above we get the claimed property.

%\label{eq:estimatedeltac}
%	\frac12\delta_\c
%	e^{-2m(2\varepsilon_E+ \varepsilon_D^\ss+\varepsilon_D^\uu+\varepsilon_D)}
%	\le \lvert f^m(\zeta)\rvert
%	\le  \delta_\c.
%\end{equation}
Arguing as above, but now interchanging the roles of the central and unstable directions, we have that for every  $p\in R^\cu_{i,m}$ the intersection 
\[
	\Delta^\uu(p)
	\eqdef R^\cu_{i,m} \cap  \widehat\cW^\uu_{f^m(z_i)}(p,\delta)
\]	 
contains a disk of inner diameter at least $\delta/4$. Therefore, by the property of the connecting times of the skeleton in Section~\ref{skeleleton}, for every $t\ge t_{\rm con}$ the set $f^t(\Delta^\uu(p))$ contains a $\uu$-disk in $V^{+\uu}_{\widehat{\mathbf C}^+_f}$ which is complete relative to $\widehat{\mathbf C}^+_f$.

\begin{remark}[Summary of the above construction]\label{rem:summary}
$\,$
\begin{itemize}
\item[(i)] For every  $p\in  f^m(\gamma_i)$ the intersection $R^\cu_{i,m} \cap  \widehat\cW^\uu_{f^m(z_i)}(p,\rho)$ contains a disk of inner diameter at least $\delta/4$.
\item[(ii)] For every  $p\in   R^\cu_{i,m} $ the curve $\alpha= R^\cu_{i,m} \cap\widehat\cW^\c_{f^m(z_i)}(p)$ has length satisfying
\begin{equation*}
%\label{e.eqcafundao}
%	\frac12\delta_\c
%	e^{-m 2\eps_2}
%	\le \lvert \alpha\rvert
%	\le  \delta_\c.
	\frac14\delta_\c K_0^{-2}
	e^{-m 2\eps_1}
	\le \lvert \alpha\rvert
	\le  \delta_\c.
\end{equation*}
\item[(iii)] The set $f^{t_{\rm con}}(\Delta^\uu (y_i))$ contains a $\uu$-disk $V^{+\uu}_{\widehat{\mathbf C}^+_f}$ which is complete relative to $\widehat{\mathbf C}^+_f$.
Observing that, for every $p\in R^\cu_{i,m}$,  the Hausdorff distance between $\Delta^\uu (p) $ and $\Delta^\uu (y_i)$ is at most $\delta_\c$ and invoking \eqref{eq:effect}, 
we have that $f^{t_{\rm con}}(\Delta^\uu (p) )$ and $f^{t_{\rm con}}(\Delta^\uu (y_i))$ are at Hausdorff distance which is smaller than the safety constant $\tau$.  Hence, by Remark~\ref{reml.minimalitysizesafety} the set $f^{t_{\rm con}} (\Delta^\uu (p))$ contains a $\uu$-disk  $\widetilde\Delta^\uu(p)$ (complete relative to ${\mathbf C}^+_f$) which is contained in $N^{+\uu}$.
Define the following set 
\[
	S_i
	\eqdef \bigcup_{p\in \gamma_i}\widetilde\Delta^\uu(p)
	\subset N^{+\uu}
	\subset \mathbf C^+_f.
\]
\item[(iv)] By construction and the forward-invariance of $\cC^\cu$ by domination, the set $S_i$ is tangent to $\cC^\cu$.  Note also that it is contained in a smooth surface, by construction. Note also that it is a $\u$-strip. 
\item[(v)] Define
\[
	\widehat R^\cu_{i,m}
	\eqdef f^{-t_{\rm con}}(S_i)
	\subset R^\cu_{i,m}.
\]
\end{itemize}
\end{remark}

\begin{lemma}\label{lem:striplength}
	For every $i\in \{1, \dots, \card \mathfrak X\}$ the (inner) width of $S_i$ can be estimated by  
\[
	\mathfrak m^{t_{\rm con}}\frac18\delta_\c K_0^{-2}
	e^{-m2\eps_1}
	\le w(S_i)
	\le 2\delta_\c \mathfrak M^{t_{\rm con}},
\]	
where $\mathfrak m$ and $\mathfrak M$ are defined in \eqref{eq:defmdf}.
\end{lemma}

\begin{proof}
Take $\alpha\in S_i$ a $(S_i,\c)$-complete curve (that is, in particular, tangent to $\cC^\c_\vartheta$). Hence, $\alpha$ is also tangent to $\cC^\cs_\vartheta$. By backward-invariance of $\cC^\cs_\vartheta$, $f^{-{t_{\rm con}}}(\alpha)$ is tangent to $\cC^\cs_\vartheta$. As, by construction, $\alpha$ is a subset of the fake invariant foliation $\widehat\cW^\cu$ which in turn is tangent to $\cC^\cu_\vartheta$, we obtain that $f^{-{t_{\rm con}}}(\alpha)$ is tangent to $\cC^\cu_\vartheta$ and consequently tangent to $\cC^\c_\vartheta$. 

Now, recall that $f^{-t_{\rm con}}(\alpha)$ is contained in $\widehat R^\cu_{i,m}$ which is in turn foliated by central fake curves tangent to $\cC^\c_\vartheta$ whose lengths are estimated in item (ii). To estimate the length of $f^{-t_{\rm con}}(\alpha)$ we only need to take into consideration the opening of the central cone field. Arguing as above, we get  
\[
	\frac18\delta_\c K_0^{-2}
	e^{-m2\eps_1}
	\le \lvert f^{-{t_{\rm con}}}(\alpha)\rvert
	\le 2\delta_\c.
\]
The lemma now follows from the definitions of $\mathfrak m$ and $\mathfrak M$ in~\eqref{eq:defmdf}.
\end{proof}

\subsubsection{Blending}\label{sss.blending}
We shall apply now the controlled expanding central covering property (Proposition~\ref{p.c.coveringproperty}) to the $\u$-strip $S_i$ which gets expanded in the central direction to obtain a  $\c$-complete $\u$-strip  of the blender-horsehoe. An important point will be to estimate both from below and above of the time needed to get this property.

 We have the following corollary from Lemma~\ref{lem:striplength} and Proposition~\ref{p.c.coveringproperty}. For the definition of the expansion constant $\lambda_{\mathrm{bh}}$ in the blender-horseshoe  see \eqref{e.nexpansionCu}.

\begin{corollary}\label{corl.minmalwidth}
There is universal constant $C>0$, such that for every $i\in\{1,\ldots,\card \mathfrak X\}$ and for every $\ell\ge \ell_{m,\eps_1}$, where
\[
	 \ell_{m,\eps_1}
	 \eqdef 
	 \left\lceil
	 	\frac{\lvert t_{\rm con}\log\mathfrak m
				-\log 8
				 +  \log\delta_\c  
				-2\log K_0
	 			-m2\eps_1\rvert}
			{\log\lambda_{\rm bh}}
						+C\right\rceil+1 
\]	
there is a subset $ S_i'\subset S_i$ such that
\begin{itemize}
\item $f^k( S_i')$ is contained in $\mathbf C^+_f$ for every $k\in\{0,\ldots,\ell\}$ and
\item $f^\ell( S_i')$ is a $\c$-complete $\u$-strip.
\end{itemize}
\end{corollary}

\subsubsection{Construction of the full rectangles}\label{sss.finalconstruction}

First, for each $\u$-strip $S_i$, we consider the corresponding set $S_i'\subset S_i$  given in Corollary~\ref{corl.minmalwidth} and define
$$
	\widetilde R_{i,m}^\cu 
	\eqdef  f^{- t_{\rm con}}(S_i')\subset \widehat R_{i,m}^\cu.
$$
We now consider the saturation by strong stable leaves of size $\delta$ of the $\cu$-rectangles above,
\[%\begin{equation}\label{eq:provisorycur0}
	R_{i,0} 
	\eqdef \bigcup_{x\in f^{-m}(\widetilde R^\cu_{i,m})} \cF^\ss (x,\delta/4)
	\subset
	R'_{i,0} 
	\eqdef \bigcup_{x\in f^{-m}(\widehat R^\cu_{i,m})} \cF^\ss (x,\delta/4).
\]%\end{equation}

\begin{remark}[The full rectangles $C_i$]\label{r.thecubesCi}
Consider the connected component of $f^{-t_{\rm con}}(R'_{i,0})\cap \mathbf C^+_f$ which contains 
$f^{-t_{\rm con}} (y_i)$ and denote it by $C_i'$ (note that in passing to the subset $\widehat R^\cu_{i,m}$ we possibly excluded the image point of $y_i$ which before served as a ``reference point'' of our construction, this is the only reason that we  consider this auxiliary set $C_i'$).  This set is contained in
$N^{+\ss}$ and is the union of  $\ss$-complete disks.
We denote by $C_i$ the connected component of $f^{-t_{\rm con}}(R_{i,0})\cap \mathbf C^+_f$ contained in $C_i'$. The full rectangle $C_i$ can be also defined as follows
then
$$
	C_i=\bigcup_{x\in R_i^\cu} \cF^\ss_{C^+_f}(x),
	\quad\text{ where }\quad
	R_i^{\cu}\eqdef f^{-(t_{\rm con}+m)}(\widetilde R^\cu_{i,m}),
$$ 
compare Figure~\ref{Fig:beleza}.
Now, recalling the definition of $\ell_{m,\eps_1}$ in Corollary~\ref{corl.minmalwidth}, let 
\begin{equation}\label{def:Nhatepsm}
	N=N_{m,\eps_1}
	\eqdef t_{\rm con}+m+t_{\rm con}+\ell_{m,\eps_1}
\end{equation}
and observe that, by its very construction, $f^N(R^\cu_i)$ is a $\c$-complete $\u$-strip.
\end{remark}

\begin{figure}[h] 
\begin{overpic}[scale=.5, %grid, tics=10
]{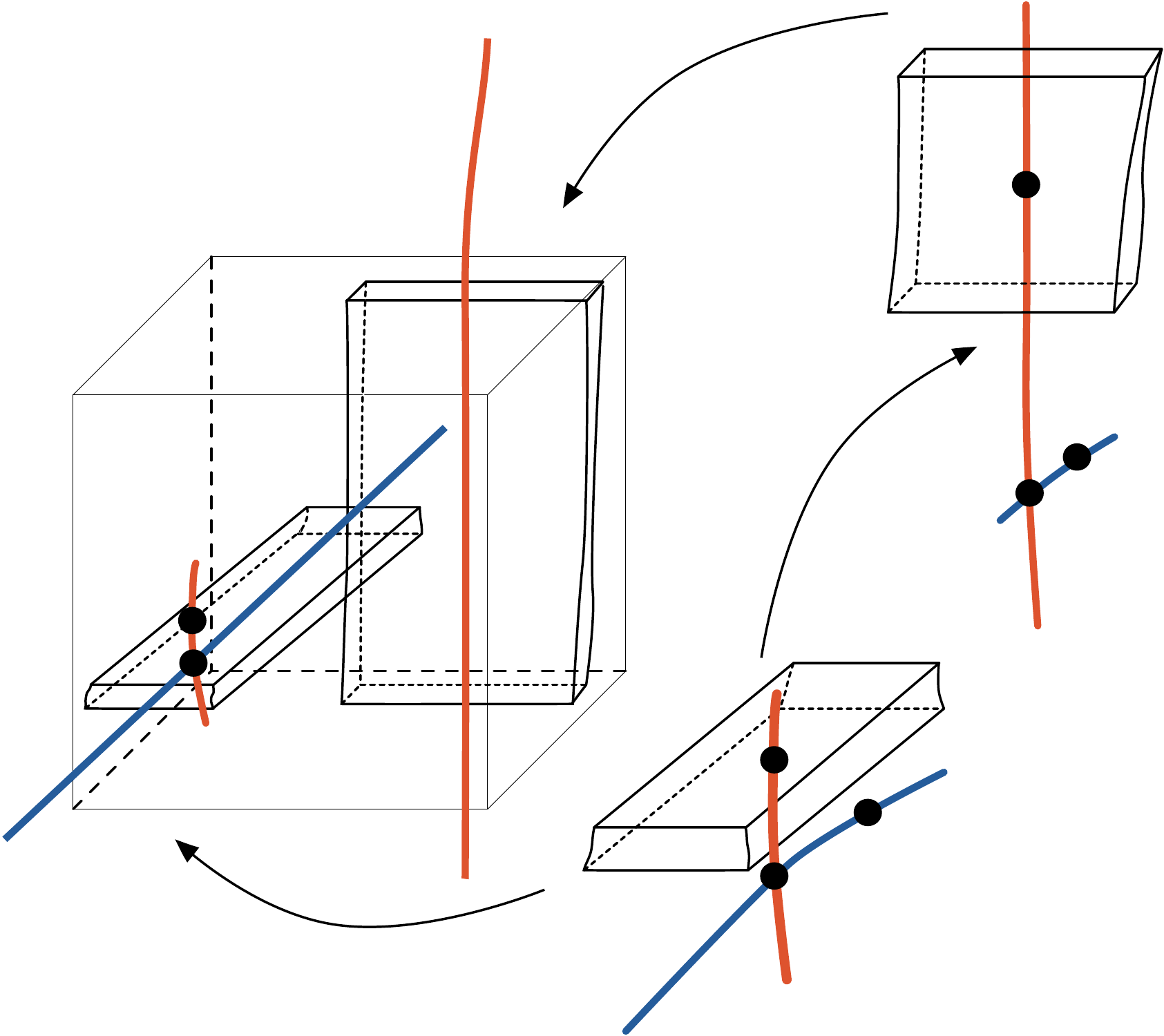}
\put(69,23){{\small{$y_i$}}}
\put(69,12){{\small{$z_i$}}}
\put(77,17){{\small{$x_i$}}}
\put(69,40){{\small{$f^m$}}}
\put(15,5){{\small{$f^{-t_{\rm con}}$}}}
\put(60,80){{\small{$f^{t_{\rm con}}$}}}
\put(9,36){{\small{$C_i$}}}
\put(83,21){{\small{$\cF^\ss(x_i)$}}}
\put(69,5){{\small{$\widehat\cW^\uu_{z_i}(z_i)$}}}
\put(20,70){{\small{$\mathbf{C}^+_f$}}}
 \end{overpic}
 \hspace{2cm}
\caption{Construction of the full rectangles $C_i$}
\label{Fig:beleza}
\end{figure}

For further reference, observe that the following property follows immediately from the definition of $\delta_\c$  in \eqref{eq:formulapetrea}.
\begin{lemma}\label{lem:katrin}
	$\ell_{m,\eps_1}\log\lambda_{\rm bh}\sim m(\sqrt{\eps_1}+2\eps_1)$.
\end{lemma}

\subsubsection{Separation and disjointness of the full rectangles}
\label{sss.separation}

\begin{lemma}[Disjointness]\label{l.bowenballs}
	The sets $C_i$ are pairwise disjoint.
\end{lemma}

\begin{proof}
	Indeed, we will prove a slightly stronger fact that any pair of points $v_i\in f^{t_{\rm con}}(C_i)$ and $v_j\in f^{t_{\rm con}}(C_j)$, $i\ne j$, is $(m,\varepsilon/3)$-separated.

Recall from the skeleton properties in Section~\ref{skeleleton} that $\{x_i\}$ is $(m,\varepsilon)$-separated.
Recall also the choices of $y_i$ and $z_i$ in Lemma~\ref{l.choosingpoints}. Observe that, by uniform contraction on strong stable manifolds  and by items 3. and 4. in Lemma~\ref{l.choosingpoints}, by we obtain that for every $\ell=0,\ldots,m$
\[
	d(f^\ell(y_i),f^\ell(x_i))
	\le d(f^\ell(y_i),f^\ell(z_i))+d(f^\ell(z_i),f^\ell(x_i))
	\le \frac\delta4+\frac\delta4
	\le 2\frac{\varepsilon}{15},
\]
where the latter follows from~\eqref{e.niteroiuff}.

\begin{claim}
For all $v\in f^{t_{\rm con}}(C_i)$,  $\max_{\ell=0,\ldots,m-1}d(f^\ell(x_i),f^\ell(v)) <{\eps}/{3}$.
\end{claim}

Since the points $\{x_i\}$ are $(m,\epsilon)$-separated this  claim implies the lemma.

\begin{proof}
Take any $v\in f^{t_{\rm con}}(C_i)$. Using the ``local product structure coordinates'' provided by the strong stable foliation and the fake invariant foliation $\widehat\cW^\cu$, we can find points  $v'\in \cF^\ss(v,\delta/4)\cap R_{i,0}$ and $v''\in\widehat\cW^\uu_{z_i}(v',\delta/4)$. Then, arguing as before, we have  for every $\ell=0,\ldots,m$
\[
	d(f^\ell(v),f^\ell(v''))
	\le 2\frac{\varepsilon}{15}.
\]
Finally, by the choice of the central fake curves of length at most $r$ in~\eqref{eq:estimatedeltac}, we obtain for every $\ell=0,\ldots,m$
\[
	d(f^\ell(v''),f^\ell(y_i))
	\le \delta_\c
	<\frac{\varepsilon}{15}.
\]
Thus, for every $\ell=0,\ldots,m$ we have $d(f^\ell(v),f^\ell(x_i))<\varepsilon/3$, which implies the claim.
\end{proof}
The proof of the lemma is now complete.
\end{proof}

\subsubsection{Definition of $\Gamma^+$}
\label{sss.Gamma}

Consider the full rectangles $C_i$, $i \in \{1,\dots, \card\mathfrak{X}\}$ 
and the return time $N=N_{m,\eps_1}$ defined in~\eqref{def:Nhatepsm} and let
\[
	\Gamma^+
	= \Gamma_{N_{m,\eps_1}}
	\eqdef \bigcup_{n=0}^{N-1}f^n(\Lambda^+),
	\quad\Lambda^+
	\eqdef \Big\{x\colon f^{kN}(x)\in\bigcup_{i=1}^{\card\mathfrak{X}}C_i
			\text{ for all }k\in\bZ\Big\}.
\]

\subsection{Hyperbolicity and entropy of the set $\Gamma^+$}
\label{s.hypGamma}

\begin{proposition}[Hyperbolicity of $\Gamma^+$]\label{p.Gammahyp}
For every $m\ge1$ large enough, the set $\Gamma^+=\Gamma_{N_{m,\eps_1}}$ is a basic set with stable index $s$ such that $f^{N_{m,\eps_1}}|_{\Lambda^+}$ is conjugate to the full shift on $\card\mathfrak{X}$ symbols. Moreover, the topological entropy of $f$ on $\Gamma^+$  satisfies
\[
 	\frac{m(h(\mu) -\varepsilon_H)- |\log L_0|}{N_{m,\eps_1}} 
	\le  h_{\rm top}(f,\Gamma^+) 
 	\le\frac{m(h(\mu) +\varepsilon_{H}) +|\log L_0|}{N_{m,\eps_1}}.
\]
\end{proposition}

\begin{proof}
As in \cite{DiaGelRam:17} it is enough for every point in the auxiliary set ${\Lambda^+}$ to estimate the finite time Lyapunov exponent corresponding to multiples of the return time $N=N_{m,\eps_1}$. Note that the points of $f^{t_{\rm con}}(C_i)$ ``shadow''  the orbit piece $x_i,\ldots,f^m(x_i)$.
By the construction of $C_i$ and by Remark~\ref{r:maisumacentral}, during this shadowing period, the derivative in the central direction for the point $f^{t_{\rm con}}(y)$, where  $y\in C_i$, satisfies for each $\ell\in\{0,\dots,m\}$,
\[
	K_0^{-1} e^{- \ell\eps_1}
	\le \lVert Df^\ell|_{E^\c_{f^{t_{\rm con}}(y)}}\rVert
	\le K_0 e^{\ell\eps_1}.
\]
Taking into account twice the transition time $t_{\rm con}$ and the blending time $\ell=\ell_{m,\eps_1}$, we obtain
\[%\begin{equation}\label{e.centralderivatives}
	K_0^{-1} e^{-m \eps_1}
		\mathfrak{m}^{2t_{\rm con}}  \lambda_{\rm bh}^{\ell}
	\le \lVert Df^{N_{m,\eps_1}}|_{E^\c_y}\rVert
\]
%\end{equation}}
where $\mathfrak m$ and $\mathfrak M$ are defined in \eqref{eq:defmdf}
and  $\lambda_{\rm bh}>1$ in \eqref{e.nexpansionCu}. 
Hence, we obtain 
\[
\begin{split} 
 	\frac{1}{N_{m,\eps_1}}\log\,\lVert Df^{N_{m,\eps_1}}|_{E^\c_y}\rVert
	&\ge \frac{-\log K_0 
			- m\eps_1
			+ 2t_{\rm con}\log \mathfrak{m} 
			+ \ell\log \lambda_{\rm bh}}
			{2t_{\rm con}+m+\ell_{m,\eps_1}}\\
	&\eqdef C(m,\eps_1).		
\end{split}
\]
By Lemma~\ref{lem:katrin}, we have $\ell\log \lambda_{\rm bh}\sim m(\sqrt{\eps_1}+2\eps_1)$.
This implies that
\[
	C(m,\eps_1)
	\sim 
	\frac{ m \left(-\eps_1+\sqrt{\eps_1}+2\eps_1 \right)}
		{m+\ell_{m,\eps_1}}.
\]
where the suppressed constants depended only on the diffeomorphism $f$, the blender-horseshoe quantifiers (expansion constants and opening of the cone fields) but neither on $m$ nor on $\eps_1$. 
Thus, the numerator is positive and if $m$ was sufficiently large, then each finite time central exponent is (uniformly) strictly positive and hence the set $\Gamma^+$ is uniformly central expanding.

We will refrain from proving the upper bound, which is analogous and anyway follows from Lemma~\ref{l.p.birkhoff} by including the continuous function $\varphi=\log\,\lVert Df|_{E^\c}\rVert$ into the family $\Phi$ considered there.

Note that, by construction, $\Gamma^+$ is compact and $f$-invariant.

 To prove  that $f^{N_{m,\eps_1}}|_{\Lambda^+}$ is conjugate to the full shift on $\card\mathfrak{X}$ symbols we use  the next claim. First, given a full rectangle $C_i$
we say that (a local) manifold $S$ tangent to cone field $\mathcal{C}^\cu$  {\emph{intersects $C_i$ in a Markovian way}}  if $S$ intersects (transversally) $\cF^\ss_{\mathbf{C}^+_f}(x)$ for all $x\in R_i^\cu$.
  
\begin{claim}
	\label{c.l.markov}
	Let 
	$S$ be a (local) manifold tangent to $\cC^\cu$ that intersects $C_i$ in a Markovian way and
	$i\in \{1,\ldots,\card\mathfrak X\}$. Let $S_i=S\cap C_i$.
     Then the set $f^{N_{m,\eps_1}}(S_i)$ intersects
	$C_j$ in a
	 Markovian way for every $j\in \{1,\ldots,\card\mathfrak X\}$.
\end{claim}
\begin{proof}
First observe that, by Remark~\ref{r.thecubesCi}, the claim is true for  $S=R^\cu_i$.
 Then observe that for any $S$ as in the claim,  the contraction on the strong stable leaves implies that the sets $f^{N_{m,\eps_1}}(S_i)$ and $f^{N_{m,\eps_1}}(R^\cu_i)$ are close. Hence, the claim follows. 
 \end{proof}

 Now, the fact that the restriction of $f^{N_{m,\eps_1}}$ to
$\Lambda^+$ is conjugate to the full shift on $\card \mathfrak X$ symbols is an easy consequence of the uniform expansion of $Df|_{\Gamma^+}$ along $E^\cu$, the uniform contraction along $E^\ss$, and Claim~\ref{c.l.markov}. This conjugation together with Lemma~\ref{lem:katrin} gives the estimates for the entropy for $f$ on $\Gamma^+$. 

The proof of the proposition is now complete.
\end{proof}

\subsection{Controlling Birkhoff averages}
\label{ss.birkhoff}

The control of the Birkhoff averages is similar to the control of the central Lyapunov exponent. The statement for the exponents would be almost immediate from the following, though we needed to show that exponents are not only close to the central exponent $\chi(\mu)$, but also positive. Moreover, we needed to show that $\Gamma^+$ was uniformly expanding in the central direction.  

Recall the choice of the number $\eps_B$ in Section~\ref{ss:prelim}.
We now prove the following result.

\begin{proposition}[Weak$\ast$ approximation]\label{p:weak}
For $m\ge1$ large enough, we have $D(\nu,\mu)<4\eps_B$
for every probability measure $\nu$ supported on $\Gamma^+$. 
\end{proposition}	

Recall the choice of the finite set $\Phi=\{\varphi_j\}$ of continuous potentials over $M$ and equation \eqref{eq:choiceofeB}.  The next lemma implies the above proposition.

\begin{lemma}\label{l.p.birkhoff}
	For every $x\in\Gamma^+$, every $\varphi_j\in\Phi$, and every $n$ large enough 
$$
	\left\lvert\frac{1}{n}\sum_{j=0}^{n-1}\varphi_k(f^j(x))-\int\varphi_k d\mu\right\rvert
	<3\eps_B.
$$
\end{lemma}

\begin{proof}%[Proof of Lemma~\ref{l.p.birkhoff}]
In a similar way as we did for the Lyapunov exponent, it suffices to prove this lemma for finite time averages relative to multiples of the return times $N=N_{m,\eps_1}$ defined in~\eqref{def:Nhatepsm} and for points $y\in\Gamma^+\cap C_i$, for any $i=1,\ldots,\card\mathfrak{X}$.

Recall that $\varepsilon_B$ and the corresponding finite family $\Phi$ was chosen in the very beginning of Section~\ref{ss:prelim}. Let us denote
\[
	\lVert\Phi\rVert
	\eqdef \max_{\varphi_j\in\Phi}\lVert\varphi_j\rVert_{C^0}.
\] 

Now recall that after this preliminary step only, the size $\rho>0$ of fake invariant center-unstable foliations was chosen. Following that, we made the choice of $\delta$ which bounded the distance of points in the ``fake unstable direction'' and the ``strong stable direction'' for any point sufficiently close to $y_i$ (and hence to $x_i$). Finally, choosing $m\ge1$ sufficiently large, we made the choice of $\delta_\c$ which bounded  the distance of points in the ``fake central direction''. 
The construction of the full rectangles $C_i$ was done in such a way that their iterations by $f^{t_{\rm con}},\ldots,f^{t_{\rm con}+m}$ have those distances controlled from the skeleton orbit $x_i,f(x_i),\ldots,f^m(x_i)$ (see the proof of Lemma~\ref{l.bowenballs}). 
Assuming that these choices where done appropriately, we can argue that we have the following estimates. Given $y\in\Gamma^+\cap C_i$, for its finite-time orbit 
\[
	\{y,f(y),\ldots,f^k(y),\ldots, f^{2t_{\rm con}+m+\ell_{m,\eps_1}}(y)\},
\]
 the  part close to the skeleton for $k=t_{\rm con},\ldots,t_{\rm con}+m$ can be estimated by
\begin{equation}\label{eq:combining}
	\left\lvert\frac{1}{m}\sum_{k=0}^{m-1}\varphi_j(f^{t_{\rm con}+k}(y))
		-\frac{1}{m}\sum_{k=0}^{m-1}\varphi_j(f^k(x_i))\right\rvert
	\le \varepsilon_B
\end{equation}
 for every $\varphi_j\in\Phi$. We refrain from giving explicit estimates. 
Hence, estimating the finite-time Birkhoff average at one return to $\Lambda^+$, we obtain
\[
	\frac{1}{N}\sum_{k=0}^{N-1} \varphi_j(f^k(y))
	\le \frac1N (2t_{\rm con}+\ell_{m,\eps_1})\lVert\Phi\rVert 
		+\frac1N\sum_{k=0}^{m-1}\varphi_j(f^{t_{\rm con}+k}(y))
\]
together with the analogous lower bound. Now, combining \eqref{eq:combining} with the property~\eqref{e.birrrrkhoff} of the skeleton, we obtain
\[\begin{split}
	\left\lvert \frac{1}{N}\sum_{k=0}^{N-1} 
		\varphi_j(f^k(y)) - \int\varphi_j\,d\mu\right\rvert
	&\le \frac{2t_{\rm con}+\ell_{m,\eps_1}}{2t_{\rm con}+m+\ell_{m,\eps_1}}\lVert\Phi\rVert
		+ \frac{m}{2t_{\rm con}+m+\ell_{m,\eps_1}}2\varepsilon_B\\
	&\eqdef 	D(m,\eps_1,\varepsilon_B).
\end{split}\]
By Lemma~\ref{lem:katrin}, we have $m/N_{m,\eps_1}\sim 1/(1+(\log\lambda_{\rm bh})^{-1}(\sqrt{\eps_1}+2\eps_1))$, which implies that when choosing $m\gg1$ sufficiently large, we get
\[
	D(m,\eps_1,\varepsilon_B)
	\le 3\eps_B.
\]
This proves the lemma.
\end{proof}

%%%%%%%%%%%%%%%%%%%%%%%%%%%%%%%%%%%%%
\section{Going from negative to positive exponents:\\ Proof of Theorem~\ref{theo:main3twin}}\label{sec:theo:main3twin}
%%%%%%%%%%%%%%%%%%%%%%%%%%%%%%%%%%%%%

In this section we will prove the following result which immediately implies Theorem~\ref{theo:main3twin}, here also notice that the case $\alpha>0>\beta$  follows considering $f^{-1}$ instead of $f$. 

\begin{theorem}\label{teo:stated}
	Assume that $f$ is a $C^1$-diffeomorphism with a partially hyperbolic splitting $TM=E^\ss \oplus E^\c\oplus E^\uu$ with three non-trivial bundles such that $E^\ss$ is uniformly contracting, $E^\c$ is one-dimensional, and $E^\uu$ is uniformly expanding such that the strong stable and the strong unstable foliations are both minimal and $f$ has an unstable blender-horseshoe for $f^{n}$ for some $n\ge1$. 
	Then for every $\mu\in \cM_{\rm erg}(f)$ with $\alpha= \chi(\mu) <0$, there is a positive constant $K(f)\ge (\log\,\lVert Df\rVert)^{-1}$ such that for every $\delta>0$, $\gamma>0$, and  $\beta>0$, there is a basic set $\Gamma$ being central expanding such that
\begin{itemize}
\item[1.] its topological entropy satisfies
	\[
		h_{\rm top}(f,\Gamma)
		\ge \frac{h(\mu)}{1+K(f)(\beta+\lvert\alpha\rvert) } - \gamma,
	\]
\item[2.] every $\nu\in \cM_{\rm erg}(f,\Gamma)$ satisfies
	\[
		 \frac \beta {1 + K(f)(\beta+\lvert\alpha\rvert)}  - \delta
		 < \chi(\nu) <
		 \frac \beta {1+\frac{1}{\log \lVert Df\rVert}(\beta+\lvert\alpha\rvert)}+ \delta,
	\]		
and
	\[
		D(\nu,\mu)
		<\frac{K(f)(\beta+\lvert \alpha\rvert)}
			{1+K(f)(\beta+\lvert \alpha\rvert)}
		+\delta.
	\]	
\end{itemize}
If $h(\mu)=0$ then $\Gamma$ is a hyperbolic periodic orbit.
\end{theorem}

We remark that the constant $K(f)$ in the above theorem is related to the minimal expansion rate in the center direction defined in~\eqref{e.nexpansionCu} in the unstable blender-horseshoe. 

The proof of Theorem~\ref{teo:stated} is very similar to the one of Theorem~\ref{teo:finally}. Hence, we only sketch the main differences. 

Consider an ergodic measure $\mu$ satisfying $\alpha\eqdef \chi(\mu)<0$ and let $\beta>0$. Denote $h= h(\mu)$. Fix as before an unstable blender-horseshoe, together with a safety domain and positive constants  $\tau,\vartheta_0,\theta>0$. Assume, for simplicity, $n^+=1$.
As in Section~\ref{ss:prelim},  fix $\varepsilon_H,\varepsilon_E,\varepsilon_B$, the constant $\vartheta$, and consider the fake invariant foliations with associate numbers $\rho>\rho_1>0$. Fix $\varepsilon_D,\varepsilon^\ss_D,\varepsilon^\uu_D,\delta_0$ and $m_0$ as before. Define $\eps_2,\eps_1$ as in~\eqref{eq:epsilonhattilde}.
Without loss of generality, we can assume that those quantifiers were chosen small enough such that we also have
\begin{equation}\label{eq:neuepss}
	-\lvert\alpha\rvert+2\varepsilon_E+\eps_D^\ss+\eps_D^\uu
	<0
	\quad\text{ and }\quad
	\eps_D<\beta.
\end{equation}

By Proposition~\ref{pro:BriKat}, there exists $\varepsilon_0>0$ so that for every $\varepsilon\in(0,\varepsilon_0)$  there are constants $K_0,L_0\geq 1$ and an integer $n_0\geq 1$ such that for every $m\geq n_0$ there exists a finite set  $\mathfrak{X}= \{x_i\}$ of $(m,\varepsilon)$-separated points so that:
\begin{itemize}
	\item $\card\mathfrak{X}\geq L_0^{-1}e^{m (h-\varepsilon_H)}$;
	\item for every $\ell=0,\ldots,m$ and every $i$ one has 
	\begin{equation}\label{e.derivadacentralneu}
	K_0^{-1}e^{-\ell(\alpha+\varepsilon_E)}\leq
	\lVert Df^\ell |_{E^\c_{x_i}} \rVert 
	\leq K_0e^{\ell(\alpha+\varepsilon_E)}
	\end{equation}
\end{itemize}   

Choose $\eps\in(0,\eps_0)$ and fix the corresponding constants $K_0,L_0,n_0$.

Fix $\delta$ as before. Fix the connecting time $t_{\rm con}$ as before. 

Finally, make an appropriate choice of $m$ sufficiently large as before. And assume also that we have 
\begin{equation}\label{eq:newneu}
	e^{-m(\beta+\lvert\alpha\rvert+\eps_2)}
	< e^{-m\beta}
	< \delta_0K_0^{-1}e^{-m\varepsilon_D}.
\end{equation}
Define 
\[%\begin{equation}\label{eq:deltacneu}
	\delta_\c
	\eqdef e^{-m\beta}.
\]%\end{equation}

Choose points $y_i,z_i$ as in Lemma~\ref{l.choosingpoints}. As in Section~\ref{sss.centralestimates}, together with~\eqref{e.derivadacentralneu}, for any point  $y\in \widehat \cW^\uu_{z_i} (z_i, \delta/4)$, for every $\ell=0,\ldots,m$ we get
\begin{equation}\label{eq:afomurlaneu}
	 K_0^{-1}e^{\ell (-\lvert\alpha\rvert-2\varepsilon_E-\eps_D^\ss-\eps_D^\uu)}\leq
	\lVert  Df^\ell |_{T_{y}\widehat\cW^\c_{z_i}} \rVert 
	\leq  K_0 e^{\ell (-\lvert\alpha\rvert+2\varepsilon_E+\eps_D^\ss+\eps_D^\uu)}.
\end{equation}
Let
\begin{equation}\label{def:rneu}
	r
	\eqdef \frac12\delta_\c=\frac12e^{-m\beta}.
\end{equation}
We now apply Proposition~\ref{p.distortiondgr} to $\vartheta$ and $\varepsilon_D$ with $\delta_0$ as in~\eqref{eq:choosedelta0} satisfying $\Mod_\vartheta(\delta_0)
	\le \varepsilon_D$ and to $\varepsilon=2\varepsilon_E+\eps_D^\ss+\eps_D^\uu$  
and $K=K_0$. Note that the hypothesis of this proposition is indeed satisfied because for all $\ell\in\{0,\ldots,m\}$ we have
\[
	\lVert Df^\ell|_{E^\c_y}\rVert
	\le K_0 e^{\ell (-\lvert\alpha\rvert+2\varepsilon_E+\eps_D^\ss+\eps_D^\uu)}
	\le K_0 e^{\ell\varepsilon_D},
\]
where we used~\eqref{eq:neuepss}. Hence, recalling that by~\eqref{eq:newneu} we have
\[
	\frac12e^{-m\beta}
	= r	
	<\delta_0K_0^{-1}e^{-m\varepsilon_D},
\]
by this proposition together with~\eqref{eq:afomurlaneu},  for every $x\in\gamma\eqdef\widehat\cW^\c_{z_i}(y,r)$ for every $\ell=0,\ldots,m$  we have 
\[%\begin{equation}\label{eq:pascoaneu}
	K_0^{-1} e^{-\ell(\lvert\alpha\rvert+ \eps_2)}
	\le \lVert Df^\ell|_{T_x\gamma}\rVert
	\le K_0 e^{-\ell(\lvert\alpha\rvert-\eps_2)}.
\]%\end{equation}

Now the construction of the auxiliary center-unstable rectangles $R^\cu_{i,0}$ is as in Section~\ref{sec:fakecurect}. Observe that they now are much ``thinner in the fake central direction'' than before and hence also always well defined (we stay always in the domain of the (locally defined) fake foliation). Recall again that our overall assumption is a partially hyperbolic splitting. In particular, the subbundles $E^\uu$ and $E^\c$ are dominated. Hence, the same arguments about the variation of the length of central curves foliating a center-unstable ``strip'' in Section~\ref{sec:fakecurect} apply. Thus, we define $R^\cu_{i,0}\eqdef\bigcup_{\gamma\in\cL}\gamma$ as before and for every curve $\gamma\in\cL$ for every $\ell\in\{0,\ldots,m\}$ we have
\[
	\frac12 r\cdot K_0^{-1}e^{-\ell(\lvert\alpha\rvert+\eps_2)}
	\le \lvert f^\ell(\zeta)\rvert
	\le 2r\cdot K_0 e^{-\ell(\lvert\alpha\rvert-\eps_2)}
\]
and by our choice of $r$ in~\eqref{def:rneu} obtain
\[
	\frac14  K_0^{-1}e^{-m(\beta+\lvert\alpha\rvert+\eps_2)}
	\le \lvert f^m(\zeta)\rvert
	\le  e^{-m(\beta+\lvert\alpha\rvert-\eps_2)}	
	<\delta_\c
	<\delta_0,
\]
where we used~\eqref{eq:newneu}.
Therefore, itens (i), (iii)--(v) in Remark~\ref{rem:summary} hold true.  
Moreover, instead of item (ii) in Remark~\ref{rem:summary} we have
\begin{itemize}
\item[(ii)] For every $p\in R^\cu_{i,m}$ the curve $\alpha=R^\cu_{i,m}\cap\widehat\cW^\c_{f^m(z_i)}(p)$ has length satisfying
\[
	\frac14 K_0^{-1}
	e^{-m(\beta+\lvert\alpha\rvert+\eps_1)}
	\le\lvert\alpha\rvert
	\le \delta_\c.
\] 
\end{itemize}
As for Lemma~\ref{lem:striplength}, we show that for every $i\in \{1, \dots, \card \mathfrak X\}$ the (inner) width of the $\u$-strip $S_i$ can be estimated as follows  
\[
	 \frac18\mathfrak m^{t_{\rm con}}K_0^{-1}e^{-m(\beta+\lvert\alpha\rvert+\eps_1)}
	\le w(S_i)
	\le \delta_\c \mathfrak M^{t_{\rm con}},
\]	
where $\frak m$ and $\frak M$ are defined in \eqref{eq:defmdf}.

Similarly as in Corollary~\ref{corl.minmalwidth} define now 
\[
	 \ell_{m,\eps_1}
	 \eqdef 
	 \left\lceil\frac{\lvert t_{\rm con}\log\mathfrak m-\log 8 -\log K_0
	 						-m(\beta+\lvert\alpha\rvert+\eps_1)\rvert}
						{\log\lambda_{\rm bh}}
						+C\right\rceil+1 .
\]	
Applying then Proposition~\ref{p.c.coveringproperty}, we get a universal constant $C>0$ such that for every $i\in\{1,\ldots,\card \mathfrak X\}$ and for every $\ell\ge \ell_{m,\eps_1}$ there is a subset $ S_i'\subset S_i$ such that
\begin{itemize}
\item $f^k( S_i')$ is contained in $\mathbf C^+_f$ for every $k\in\{0,\ldots,\ell\}$ and
\item $f^\ell( S_i')$ is a $\c$-complete $\u$-strip.
\end{itemize}
Observe that 
\begin{equation}\label{eq:invocar}
	\ell_{m,\eps_1}\log\lambda_{\rm bh}\sim m(\beta+\lvert\alpha\rvert+\eps_1).
\end{equation}

The construction of the full rectangles $C_i$ is again as in Section~\ref{sss.finalconstruction}. Analogously, one verifies that they are pairwise disjoint. 
Define as before $N_{m,\eps_1}\eqdef 2t_{\rm con}+m+\ell_{m,\eps_1}$. One defines the set $\Gamma^+$ as in Section~\ref{sss.Gamma}.
As in Remark~\ref{r:maisumacentral}, it is a consequence that the point $f^{t_{\rm con}}(y)$, for any $y\in C_i$, satisfies for each $k\in\{0,\dots,m\}$,
\[
	K_0^{-1} e^{- k(\lvert\alpha\rvert+ \eps_1)}
	\le \lVert Df^k|_{E^\c_{f^{t_{\rm con}}(y)}}\rVert
	\le K_0 e^{-k (\lvert\alpha\rvert- \eps_1)}.
\]
Taking into account twice the transition time $t_{\rm con}$ and the blending time $\ell=\ell_{m,\eps_1}$, we obtain
\[%\begin{equation}\label{e.centralderivatives}
	K_0^{-1} e^{-m(\lvert\alpha\rvert+ \eps_1)}  \mathfrak{m}^{2t_{\rm con}}  \lambda_{\rm bh}^\ell
	\le \lVert Df^{N_{m,\eps_1}}|_{E^\c_y}\rVert.
\]%\end{equation}
Hence, we obtain
\[
\begin{split} 
 	\frac{1}{N_{m,\eps_1}}\log\,&\lVert Df^{N_{m,\eps_1}}|_{E^\c_y}\rVert\\
	&\ge \frac{-\log K_0 -m(\lvert\alpha\rvert+ \eps_1) 
			+2t_{\rm con}\log \mathfrak{m} 
			+\ell_{m,\eps_1}\log \lambda_{\rm bh}}
			{2t_{\rm con}+m+\ell_{m,\eps_1}}\\
	&\eqdef C(m,\eps_1)	.	
\end{split}
\]
Finally, using now~\eqref{eq:invocar}, note that
\[
	C(m,\eps_1)
	\sim \frac{-m(\lvert\alpha\rvert+ \eps_1) +m(\beta+\lvert\alpha\rvert+\eps_1)}
		{m+\frac{1}{\log\lambda_{\rm bh}}m(\beta+\lvert\alpha\rvert+\eps_1)}
	\sim\frac{\beta}{1+\frac{1}{\log\lambda_{\rm bh}}(\beta+\lvert\alpha\rvert)}
	>0	.
%	\ge\frac{\beta-\eps_D^\ss}
%		{1+\frac{1}{\log\lambda_{\rm bh}}(\beta+\lvert\alpha\rvert)}.	
\]
This sketches the proof that $\Gamma^+$ is central expanding and provides the claimed lower bound for the Lyapunov exponents of any orbit and hence any measure supported on $\Gamma^+$. The upper bound is analogous.

In order to estimate the entropy of $f$ on $\Gamma^+$, observe 
\[
	h_{\rm top}(f,\Gamma^+)
	\sim\frac{\log\card\mathfrak X}{N_{m,\eps_1}}
	\sim \frac{m(h(\mu)-\eps_H)}{m+\ell_{m,\eps_1}}  
	\sim\frac{h(\mu)}
	{1+\frac{1}{\log\lambda_{\rm bh}}(\beta+\lvert\alpha\rvert)}.	
\]

Finally, to argue about the weak$\ast$ approximation of the given measure $\mu$ as in Section~\ref{ss.birkhoff}, observe that $\Gamma^+$ is built very close to the orbit pieces of $\mu$-generic points provided from the skeleton. Thus, orbital measures in the skeleton arbitrarily well approximate  $\mu$. The return time from one Markov rectangle into another one is $N_{m,\eps_1}= 2t_{\rm con}+m+\ell_{m,\eps_1}$. Thus, any ergodic invariant measure supported on $\Gamma^+$ has generic orbits that always stay a fraction $m/(2t_{\rm con}+m+\ell_{m,\eps_1})$ of its time as close to $\mu$ as desired. Hence, the weak$\ast$-deviation from $\mu$ is of order
\[
	\frac{\ell_{m,\eps_1}}{m+\ell_{m,\eps_1}}
	\sim \frac{\frac{1}{\log\lambda_{\rm bh}}m(\beta+\lvert\alpha\rvert)}
			{m+\frac{1}{\log\lambda_{\rm bh}}m(\beta+\lvert\alpha\rvert)}.
\] 
Letting $K(f)=1/\log\lambda_{\rm bh}$, finishes the sketch of the proof of the theorem.
\hfill\qed

%%%%%%%%%%%%%%%%%%%%%%%%%%%%%%%%%%%%%
\section{Arcwise connectedness: Proof of Theorem \ref{teocor:simplleexx}}\label{sscor:simplleexx}
%%%%%%%%%%%%%%%%%%%%%%%%%%%%%%%%%%%%%

For completeness, recall that the \emph{homoclinic class} of  a hyperbolic periodic point $Q_f$ of a diffeomorphism $f$, denoted by $H(Q_f,f)$, is the closure of the set of  transverse intersection points of the stable and unstable manifolds of the orbit of $Q_f$. Two hyperbolic periodic  points $P_f$ and $Q_f$ of $f$ are \emph{homoclinically related} if the stable and unstable manifolds of their orbits  intersect cyclically and transversely. The homoclinic class of $Q_f$ can  also be defined as the closure of the periodic points of $f$ that are homoclinically related to $Q_f$. A homoclinic class is a \emph{transitive set} whose periodic points  form a dense subset of it. The \emph{stable index} of a hyperbolic periodic point is the dimension of its stable bundle. Note that, in our partially hyperbolic context with one-dimensional center bundle, there are only two types of hyperbolic periodic points: those with stable index $s$ and those with $s+1$. The fact that the central bundle is one-dimensional also forces the intersection between invariant manifolds of hyperbolic periodic points of the same index to be  transverse.

Now recall that by \cite[Proposition 7.1]{BocBonDia:18} there is an open and dense subset of $\cRTPH^1(M)$ consisting of diffeomorphisms $f$ having a pair of saddles $P_f$ and $Q_f$ of stable index $s$ and $s+1$, respectively, such that the homoclinic classes satisfy 
\begin{equation}\label{eq:noproblems}
	H(P_f,f)=M=H(Q_f,f).
\end{equation}	 
This result just summarizes previous ones in \cite{BonDiaUre:02,BonDiaPujRoc:03,RodRodUre:07}. Moreover, the minimality of both the strong stable and the strong unstable foliations together with the partial hyperbolicity immediately imply that every pair of  saddles of the same index is homoclinically related. In this case, the fact that the homoclinic class is isolated is immediate and thus, we are in the setting of \cite{GorPes:17}. Thus, in what follows we will assume that the set $\cORTPH^1(M)$ additionally satisfies \eqref{eq:noproblems} and fix $f\in\cORTPH^1(M)$. 

Recall that the \emph{convex hull} of a set $\cN\subset\cM(f)$ is the smallest convex set containing $\cN$, denoted by $\conv(\cN)$, and that the \emph{closed convex hull} of $\cN$ is the smallest closed convex set containing $\cN$, denoted by $\cconv(\cN)$. By \cite[Theorem 5.2 (i)--(ii)]{Sim:11}, we have $\overline{\conv(\cN)}=\cconv(\cN)$, where $\overline{\cN}$ denotes the weak$\ast$ closure of $\cN$. Our hypotheses imply that we can invoke \cite[Theorem 2]{BocBonGel:}, that is, every $\mu\in\conv(\cM_{\rm erg,<0}(f))$ can be approximated in the weak$\ast$ topology by ergodic measures (in $\cM_{\rm erg,<0}(f)$). In other words,  we have
\[
	\cM_{\rm erg,<0}(f)
	\subset	\conv(\cM_{\rm erg,<0}(f))
	\subset\overline{\cM_{\rm erg,<0}(f)}
\]
(the first inclusion is trivial) and hence, taking closures and applying the above comment, we can conclude 
\[
	\cconv(\cM_{\rm erg,<0}(f))
	=\overline{\conv(\cM_{\rm erg,<0}(f))}
	=\overline{\cM_{\rm erg,<0}(f)}.
\]	 
Analogously for $\cM_{\rm erg,>0}(f)$.
On the other hand, by Theorem~\ref{t.approx}, we have
\[\begin{split}
	{\cM_{\rm erg,0}(f)}	
	&\subset \overline{\cM_{\rm erg,<0}(f)}\cap\overline{\cM_{\rm erg,>0}(f)}\\
	&=\cconv(\cM_{\rm erg,<0}(f))\cap\cconv(\cM_{\rm erg,>0}(f)),
\end{split}\]
which proves the first claim in the corollary.

We now prove that the set of measures $\cM_{\rm erg,<0}(f)$ is arcwise connected, the proof for $\cM_{\rm erg,>0}(f)$ is analogous. Here we will largely follow arguments in \cite{GorPes:17}, see also the presentation in \cite[Section 3.1]{DiaGelRam:17b}. 
Take any pair of measures $\mu^0,\mu^1\in\cM_{\rm erg,<0}(f)$. By Corollary~\ref{cor:2}, each $\mu^i$ is accumulated by a sequence of hyperbolic periodic measures $\nu_n^i\in\cM_{\rm erg,<0}(f)$ supported on periodic points $P_n^i$, $i=0,1$ respectively. By assumption, these points are homoclinically related and hence, there exists a basic set $\Gamma=\Gamma(\mu^0,\mu^1)$ containing them. Recall from \cite{LinOlsSte:78,Sig:74} that  $\cM(f,\Gamma)$ is a Poulsen simplex. Hence  there is a continuous arc $\mathfrak m_0\colon[1/3,2/3]\to\cM_{\rm erg}(f,\Gamma)\subset \cM_{\rm erg,<0}(f)$ joining the measures $\nu_1^0=\mathfrak m_0(1/3)$ and $\nu_1^1=\mathfrak m_0(2/3)$. For any pair of measures $\nu_n^0,\nu_{n+1}^0$, the same arguments apply and, in particular, there exists a continuous arc $\mathfrak m_n^0\colon[1/3^{n+1},1/3^n]\to\cM_{\rm erg,<0}(f)$ joining the measure $\nu_n^0=\mathfrak m_n^0(1/3^{n+1})$ with $\nu_{n+1}^0=\mathfrak m_n^0(1/3^n)$. 
Using those arcs and concatenating   (appropriate parts of) their domains, we obtain  an arc $\bar {\mathfrak m}_n^0\colon[1/3^{n+1},1/3]\to \cM_{\rm erg,<0}(f)$ joining $\nu_{n+1}^0=\bar{\mathfrak m}_n^0(1/3^{n+1})$ and $\nu_{1}^0=\bar{\mathfrak m}_n^0(1/3)$.
The same applies to the measures $\nu_n^1$, defining arcs $\bar{ \mathfrak m}_n^1 \colon[1-1/3^n,2/3]\to\cM_{\rm erg,<0}(f)$ joining $\nu_{n+1}^1$ and $\nu_{1}^1$. 
 Define now $\mathfrak m_\infty|_{(0,1)}\colon(0,1)\to\cM_{\rm erg,<0}(f)$ by concatenating (appropriate parts of) the domains of those arcs. We  complete the definition of an arc $\mathfrak m_\infty\colon[0,1]\to\cM_{\rm erg,<0}(f)$ by letting $\mathfrak m_\infty(0)=\lim_{n\to\infty}\bar {\mathfrak m}_n^0(1/3^n)$ and $\mathfrak m_\infty(1)=\lim_{n\to\infty}\bar {\mathfrak m}_n^1(1-1/3^n)$. By definition, $\mathfrak m_\infty$ joins $\mu^0$ and $\mu^1$. Note that in the last step we assume that $\mu^0,\mu^1$ do not belong to the image of $\mathfrak m_\infty$, if one of these measures does belong it is enough to cut the domain of definition of $\mathfrak m_\infty$ appropriately.

The analogous construction can be done to construct an arc connecting any measure in $\cM_{\rm erg,0}(f)$ to any measure in $\cM_{\rm erg,>0}(f)$ using  Theorem \ref{t.approx} and then the second part of Corollary~\ref{cor:2}.
\hfill\qed
 
\bibliographystyle{plain}
\bibliography{refs}
\end{document}